\documentclass[11pt]{amsart}

\usepackage{amsmath, amssymb, amsthm, amsfonts, fullpage}
\usepackage[mathcal]{eucal}

\input xy
\xyoption{all}

\usepackage{tikz-cd}
\usepackage{dynkin-diagrams}
\usepackage{array}   
\newcolumntype{C}{>{$}c<{$}} 
\usepackage{diagbox}

\usepackage{hyperref}
\hypersetup{hidelinks}

\numberwithin{equation}{section}

\theoremstyle{plain}

\newtheorem{theorem}{Theorem}[section]
\newtheorem{lemma}[theorem]{Lemma}

\newtheorem{prop}[theorem]{Proposition}
\newtheorem{cor}[theorem]{Corollary}

\theoremstyle{definition}

\newtheorem{defn}[theorem]{Definition}

\newtheorem{exam}[theorem]{Example}

\newtheorem{remark}[theorem]{Remark}

\def\AA{\mathbb{A}}

\def\CC{\mathbb{C}}

\def\GG{\mathbb{G}}

\def\QQ{\mathbb{Q}}
\def\RR{\mathbb{R}}

\def\ZZ{\mathbb{Z}}

\newcommand\cA{\mathcal{A}}
\newcommand\cB{\mathcal{B}}

\newcommand\cD{\mathcal{D}}

\newcommand\cI{\mathcal{I}}
\newcommand\cJ{\mathcal{J}}

\newcommand\cL{\mathcal{L}}
\newcommand\cM{\mathcal{M}}

\newcommand\cO{\mathcal{O}}

\newcommand\cS{\mathcal{S}}
\newcommand\cT{\mathcal{T}}

\newcommand\cV{\mathcal{V}}

\newcommand\cY{\mathcal{Y}}

\def\bC{\mathbf{C}}

\def\bO{\mathbf{O}}
\def\bP{\mathbf{P}}

\newcommand\frc{\mathfrak{c}}

\newcommand\frg{\mathfrak{g}}

\newcommand\frt{\mathfrak{t}}

\newcommand\ab{\textup{ab}}

\newcommand\alg{\textup{alg}}

\newcommand\an{\textup{an}}

\newcommand{\codim}{\textup{codim}}

\newcommand\Gal{\textup{Gal}}

\newcommand{\Gr}{\textup{Gr}}
\newcommand{\gr}{\textup{gr}}

\newcommand\Lie{\textup{Lie}}

\newcommand{\ord}{\textup{ord}}

\newcommand\pr{\textup{pr}}

\newcommand\pt{\textup{pt}}

\newcommand{\red}{\textup{red}}

\newcommand\rk{\textup{rk}}
\newcommand\rs{\textup{rs}}

\newcommand\sgn{\textup{sgn}}

\newcommand\Spec{\textup{Spec}\,}

\newcommand\Sym{\textup{Sym}}

\newcommand{\val}{\textup{val}}

\newcommand\Hom{\textup{Hom}}

\newcommand\Map{\textup{Map}}

\newcommand\GL{\textup{GL}}

\newcommand\Sp{\textup{Sp}}

\newcommand\ra{\rightarrow}

\newcommand{\Gm}{\GG_m}

\newcommand\quash[1]{}

\newcommand{\ov}{\overline}

\newcommand{\cohog}[2]{\textup{H}^{#1}({#2})}     
\newcommand{\cohoc}[2]{\textup{H}_{c}^{#1}({#2})}     

\renewcommand\a\alpha
\renewcommand\b\beta
\newcommand\G\Gamma
\newcommand\g\gamma
\renewcommand\d\delta
\newcommand\D\Delta
\newcommand{\e}{\epsilon}

\renewcommand{\th}{\theta}

\newcommand{\s}{\sigma}

\newcommand{\z}{\zeta}

\renewcommand{\l}{\lambda}

\usepackage{graphicx}

\newcommand{\et}{\textup{ét}}

\newcommand{\init}{\operatorname{in}}
\newcommand{\Trop}{\operatorname{Trop}}
\newcommand{\trop}{\operatorname{trop}}

\newcommand{\Ctalg}{\CC\langle t\rangle}
\newcommand{\Adm}{\operatorname{Adm}}

\newcommand{\Sch}{\textup{Sch}}
\newcommand{\Spf}{\textup{Spf}}
\newcommand{\Spa}{\textup{Spa}}
\newcommand{\rig}{\textup{rig}}

\author{Konstantin Jakob}
\address{(KJ) TU Darmstadt, Fachbereich Mathematik, Schlossgartenstraße 7, 64289 Darmstadt, Germany} 
\email{jakob@mathematik.tu-darmstadt.de}

\author{Masoud Kamgarpour}
\address{(MK) School of Mathematics and Physics, The University of Queensland, Australia} 
\email{masoud@uq.edu.au}

\author{Ian Le}
\address{(IL) Mathematical Sciences Institute Australian National University, Canberra ACT
2601, Australia} 
\email{ian.le@anu.edu.au}

\author[Minh-T\^{a}m Trinh]{Minh-T\^{a}m Quang Trinh}
\address{(MT) Department of Mathematics, Howard University, Washington, DC 20059}
\email{minhtam.trinh@howard.edu}


\title{Algebraic loops, braids, and tropicalization}

\begin{document}

\begin{abstract} Let $Y$ be a smooth complex algebraic variety. We develop a tropical framework for studying algebraic loops, i.e., free homotopy classes in $Y^{\an}$ associated with formal loops $\Spec\CC(\!(t)\!)\to Y$. When $Y$ is very affine, every formal loop determines an integral point of $\Trop(Y)$. For schön varieties, we prove that the associated algebraic loop is determined by this tropical point together with a connected component of the corresponding initial variety. Consequently, when all initial varieties are connected, integral tropical points classify algebraic loops. We apply this theory to the case when $Y$ is a hyperplane arrangement complement. In this case, we identify the initial varieties with complements of graded arrangements and obtain explicit representatives for algebraic loops as products of relative full twists. In type $A$, these representatives recover pure algebraic braids. For root arrangements, our construction gives a tropical interpretation of the split root valuation data of Goresky, Kottwitz, and MacPherson. Finally, we prove that the cohomology of braid varieties associated with positive algebraic braids depends only on their root valuation data.
\end{abstract}

\maketitle

\tableofcontents

\section{Introduction}

Loops are among the most fundamental objects in topology. In algebraic geometry, however, one naturally encounters \emph{formal loops}, namely morphisms
\[
f:\Spec\CC(\!(t)\!)\longrightarrow Y,
\]
rather than continuous maps from the circle. Formal loops arise in many areas of algebraic geometry and representation theory, including motivic integration, affine Springer theory, the study of irregular connections, and algebraic braids. In these settings, one often attaches valuation data to a formal loop: pole orders of eigenvalues, root valuation data, fission trees, or the contact order data of Puiseux branches of plane curves.

The purpose of this paper is to show that, in the split setting, these valuation data and the associated topological monodromy are governed by a single tropical mechanism. We associate with each formal loop a canonical free homotopy class in $Y^{\an}$, called its algebraic loop, and relate this class to the tropicalization of the formal loop. For schön very affine varieties, the algebraic loop is determined by the integral tropical point together with the connected component of the corresponding initial variety, which is the space of leading term coefficients of formal loops. For hyperplane arrangement complements, this abstract classification becomes an explicit product formula, with applications to algebraic braids, root valuation data, and braid varieties.

\subsection{Main results} 
Let $Y$ be a smooth connected complex algebraic variety. Then we can associate to every formal loop
$f$ 
a canonical free homotopy class
\[
\gamma_f\in[S^1,Y^{\an}],
\]
called the \emph{algebraic loop} determined by $f$. Roughly speaking, this loop is obtained by approximating $f$ by analytic maps on sufficiently small punctured discs and restricting these maps to the boundary circle. 
\begin{theorem}[Algebraic loop map, see Theorem~\ref{thm:Lalg}]
There is a natural map \[
\cL_{\alg}:Y(\CC(\!(t)\!))\longrightarrow[S^1,Y^{\an}],\qquad f\mapsto \gamma_f, 
\]
compatible with the \'etale loop associated to $f$. 
\end{theorem}
We denote the image of this map by $[S^1,Y^{\an}]^{\alg}$. The main goal of this paper is to understand the algebraic loop map using tropical data. To this end, we first need an understanding of the relationship between the loop space and the tropicalization of $Y$.

We will focus on the case where $Y$ is a very affine variety, i.e. an affine variety whose coordinate ring is generated by units. Such varieties possess a canonical intrinsic torus whose character lattice is the group of units modulo constants. Consequently they admit canonical tropicalizations $\mathrm{Trop}(Y)$.  Every formal loop $f$ determines a tropical point
\[
w_f\in\Trop(Y)\cap X_*(T),
\]
which is characterized by $\langle \chi, w_f \rangle = \val_{t}(f^*\chi)$ for every character $\chi$ of $T$. The tropical point records the orders of zeros or poles of units along the formal loop. For every tropical point
\[
w\in\Trop(Y)\cap X_*(T),
\]
we define the corresponding \emph{valuation stratum}
\[
L_wY\subset LY
\]
inside the formal loop space of $Y$. Although $L_wY$ is defined as the \emph{set} of formal loops with valuation $w$, we show that it carries a canonical scheme structure. For this, one uses the so-called Gröbner degeneration $\cY_{w} \to \AA^1$, which is a flat degeneration with general fiber $Y$ and special fiber given by the initial variety $\init_{w}Y$. More precisely, we have: 

\begin{theorem}[Tropicalization and the loop space, see Theorem \ref{thm:tropical-fiber-main}] Let $\mathcal Y_w$ denote the Gröbner degeneration of $Y$ associated with $w$ and $\widehat{\mathcal Y}_w$ the corresponding formal degeneration. Then we have
\[
L_wY\cong\Gr_\infty(\widehat{\mathcal Y}_w),
\]
where $\Gr_\infty$ is the infinite Greenberg scheme. Moreover, there is a natural morphism $L_{w}Y \to \init_{w}(Y)$, called initial term morphism, which is a locally trivial infinite affine space bundle when $Y$ and $\init_w{Y}$ are smooth.
\end{theorem} 

When the Gröbner degeneration is smooth, the valuation stratum is an infinite affine-space bundle over the initial variety. Consequently, the topology of the valuation strata is governed by the topology of initial varieties. This allows us to identify the connected components of the loop space
with tropical points enhanced by connected components of the corresponding
initial varieties. Consequently, an algebraic loop depends only on this enhanced tropical datum. When all initial
degenerations are connected, it depends only on the tropicalization:

\begin{theorem}[Tropical classification of algebraic loops, see Theorem \ref{thm:trop-classification}]
Let $Y$ be a schön very affine variety. The algebraic loop associated with a formal loop depends only on its tropicalization together with the connected component of the corresponding initial variety. 

If $Y$ is wunderschön, i.e. every initial variety is connected, then the algebraic loop map induces a natural bijection
\[
\pi_0(LY)
\xrightarrow{\sim}
\Trop(Y)\cap X_*(T)
\xrightarrow{\sim}
[S^1,Y^{\an}]^{\alg}.
\]
\end{theorem}

\begin{remark} In the final stages of the preparation of the manuscript, we became aware that the existence of the algebraic loop map in Theorem 1.1 follows from the work of Achinger–Talpo \cite{AT}, while the first part of Theorem 1.2 is already present in \cite{KU}. See Remarks \ref{AT} and \ref{r:KU} for more details. We have retained our formulations and proofs because they fit naturally into the framework developed here and lead to the tropical classification in Theorem 1.3 and the applications below. 
\end{remark}

\subsection{Application to hyperplane arrangements} 
Let $V$ be a complex vector space, let $\cA$ be an essential
linear hyperplane arrangement in $V$, and set
\[
Y:=V\setminus\bigcup_{H\in\cA}H.
\]
Then $Y$ is a smooth very affine variety. Our aim is to give an explicit description of algebraic loops in $Y$.

To state the theorem, we need some notation.
After choosing a base point
$y\in Y$, we write
$P:=\pi_1(Y^{\an},y)$
for the arrangement group of $(V,\cA)$. The scaling action of $\CC^\times$ on $Y$ defines the full twist
\[
\pi_P\in P,
\qquad
e^{i\theta}\longmapsto e^{i\theta}y.
\]
More generally, each flat $X$ in the intersection lattice of $\cA$
determines an intersection subgroup $P_X\subset P$ and a corresponding
full twist $\pi_{P_X}\in P_X$. If $P'\subset P''$ are nested intersection
subgroups, we define the relative full twist by
\[
\pi_{P''/P'}:=\pi_{P''}\pi_{P'}^{-1}.
\]
The precise construction of the intersection subgroups and their full
twists is recalled in Section~\ref{ss:intersection-sub}.

The intrinsic torus of $Y$ is naturally identified with
$\Gm^{\cA}$. Accordingly, the tropicalization of a formal loop
$f\in Y(F)$ may be regarded as the \emph{hyperplane valuation function}
\[
d_f:\cA\longrightarrow\ZZ,
\qquad
H\longmapsto \val_t\bigl(\ell_H(f)\bigr),
\]
where $\ell_H$ is any defining linear form for $H$. This definition is
independent of the choice of $\ell_H$, and Proposition
\ref{prop:tropwinding} identifies $d_f(H)$ with the winding number of
$\gamma_f$ around $H$.

Write
\[
\{d_f(H)\mid H\in\cA\}=\{d_1<\cdots<d_r\}.
\]
The subarrangements
\[
\cA_{>d_j}:=
\{H\in\cA\mid d_f(H)>d_j\},
\qquad 1\leq j<r,
\]
determine a decreasing filtration
\[
P=P_0\supset P_1\supset\cdots\supset P_{r-1}
\]
by intersection subgroups. This filtration is an arrangement-theoretic
generalization of the fission sequences appearing, for example, in
\cite{DRT}. 

\begin{theorem}[Explicit representatives for algebraic loops,
Theorem \ref{thm:braidrep}]
The
conjugacy class of the algebraic loop $\gamma_f$ is represented by
\[
\pi_{P_0/P_1}^{d_1}
\pi_{P_1/P_2}^{d_2}
\cdots
\pi_{P_{r-2}/P_{r-1}}^{d_{r-1}}
\pi_{P_{r-1}}^{d_r}.
\]
\end{theorem}

An important special case is the braid arrangement of type $A_{n-1}$.
Let
\[
V=\CC^n/\CC(1,\ldots,1),
\qquad
\cA=\{x_i=x_j\mid i\neq j\},
\]
and let $W=S_n$ act by permuting the coordinates. The complement $Y$ is
homotopy equivalent to the ordered configuration space of $n$ distinct
points in $\CC$. By the classical theorem of Artin, there is an
isomorphism
\[
\pi_1(Y^{\an})\cong P_n,
\]
where $P_n$ is the pure braid group on $n$ strands.

By a \emph{pure algebraic braid} we mean the local braid monodromy of a
reduced plane curve singularity all of whose branches are smooth. Its
closure is the corresponding algebraic link; see
\cite{EisenbudNeumann}. Under the above identification, we obtain a
natural correspondence
\[
\left\{
\begin{array}{c}
\text{positive algebraic loops}\\
\text{in }Y
\end{array}
\right\}
\longleftrightarrow
\left\{
\begin{array}{c}
\text{pure algebraic}\\
\text{braids}
\end{array}
\right\}.
\]
Here, an algebraic loop is called \emph{positive} if its hyperplane
valuation function satisfies
\[
d_f(H)>0
\qquad\text{for every }H\in\cA.
\]

Returning to an arbitrary essential linear hyperplane arrangement, the
group $\pi_1(Y^{\an})$ is its arrangement group. Positive algebraic
loops in $Y$ may therefore be regarded as arrangement-theoretic
analogues of pure algebraic braids. The extension to the non-pure
setting is discussed further below.

\subsection{Application to braid varieties}

Let $G$ be a simple complex algebraic group with Cartan subalgebra
$\frt$, root system $\Phi$, and Weyl group $W$. Consider the hyperplane complement 
\[
Y:=\frt^{\rs}
=
\frt\setminus\bigcup_{\alpha\in\Phi}\ker(\alpha).
\]
Then
\[
P_W:=\pi_1(Y)
\qquad\text{and}\qquad
B_W:=\pi_1(Y/W)
\]
are respectively the pure braid group and the braid group associated
with $W$.

Let $\beta\in B_W^+$ be a positive braid, and choose a
factorization
\[
\beta=\widetilde w_1\cdots\widetilde w_m,
\qquad w_i\in W,
\]
where $\widetilde w_i$ denotes the positive lift of $w_i$. Let
$\cB$ be the flag variety of $G$, and let
$\bO(w)\subset\cB\times\cB$ denote the $G$-orbit corresponding to
$w\in W$. The braid variety associated with $\beta$ is
\[
Z(\beta)
=
\left\{
(B_0,\ldots,B_m,g)\in\cB^{m+1}\times G
\;\middle|\;
\begin{array}{l}
(B_i,B_{i+1})\in\bO(w_{i+1}),\quad 0\leq i<m,\\
{}^gB_0=B_m
\end{array}
\right\}.
\]
Thus, $Z(\beta)$ parametrizes chains of Borel subgroups in the relative
positions prescribed by $w_1,\ldots,w_m$, together with an element of
$G$ identifying the two endpoints.

A formal loop $f\in Y(\CC(\!(t)\!))$ determines a pure algebraic braid
\[
\beta_f:=\gamma_f\in P_W\subset B_W.
\]
When its root valuations are positive, $\beta_f$ is represented by a
positive braid. Combining our tropical classification with the
horocycle correspondence and the categorical trace formalism, we prove
the following.

\begin{theorem}[Corollary \ref{c:braidVar}]
The cohomology of the braid variety attached to a
positive algebraic braid depends only on its root valuation data.
\label{thm:braid-variety-valuation}
\end{theorem}

This provides an instance of the general principle (discussed below) that geometric
objects associated with formal loops are controlled by their valuation
data.

\subsection{Motivation}

Our motivation comes primarily from the study of irregular connections,
affine Springer fibers, and geometric representation theory. Braids
arise naturally throughout these subjects. For irregular connections,
they encode the combinatorics of
Stokes filtrations \cite{Boalch1,BoalchPoisson,BoalchYamakawa,BDR0}.
This leads to moduli spaces described by braid varieties, which may be
viewed as characteristic-zero analogues of Deligne--Lusztig varieties,
with monodromy around the puncture replacing the action of Frobenius.

Braid varieties also provide a connection with knot theory. In type
$A$, their cohomology is closely related to the
Khovanov--Rozansky homology of braid closures. Of particular interest
are algebraic links, namely the links of plane curve singularities. The theory of algebraic braids as developed by the last-named author gives a braid-theoretic description of this class of links \cite{Trinh}.

A common theme in these settings is that geometric objects associated
with formal loops often depend only on valuation data. Wild character
varieties fit into locally trivial families when the relevant pole-order
data are preserved \cite{BDR}, while a conjecture of Goresky, Kottwitz,
and MacPherson predicts that the cohomology of an affine Springer fiber
depends only on its root valuation data. A proof of this conjecture for $\GL_d$ has recently been announced by
Chen \cite{ChenRootValuation}.

The starting point of this paper is the observation that root valuation
functions have a natural tropical interpretation: in the split case,
they are precisely the integral points in the tropicalization of the
root hyperplane arrangement complement. Tropical geometry therefore
provides a natural framework for organizing these valuation data. Our
results show that it also controls the algebraic loops associated with
formal loops in hyperplane arrangement complements and naturally
recovers structures such as fission sequences and fission trees.

\subsection{Future directions and relation to root valuation data}

In this paper, we restrict our attention to the split, or pure, case.
In future work, we plan to study finite quotients $X=Y/\Gamma$, where
$Y$ is very affine and $\Gamma$ acts freely on $Y$. A formal loop in
$X$ need not lift to a formal loop in $Y$ over the original ground
field. After passing to Puiseux series, however, such a lift exists and
determines two pieces of information: a point of $\Trop(Y)$ and a
finite monodromy element in $\Gamma$. Informally, these record the
valuation and twisting data of the loop. They may be packaged naturally
as an object of the inertia stack
\[
I([\Trop(Y)/\Gamma])
\]
in a suitable category of cone stacks \cite{CCUW}. When $\Gamma=W$ is a Weyl group and $Y$ is the corresponding
hyperplane arrangement complement, this combined tropical and monodromy
datum recovers the root valuation data of Goresky, Kottwitz and
MacPherson \cite{GKM}. This perspective gives a tropical interpretation
of their non-emptiness conditions and suggests a relationship with
twisted admissible deformation spaces \cite{DRY}. We expect, more
generally, that algebraic loops in $Y/\Gamma$ depend up to homotopy only
on this combined valuation and monodromy datum. We will develop this
picture in future work.

\subsection{Structure of the paper}

In Section \ref{s:algloop} we construct the algebraic loop map and explain its relation to the work of Achinger--Talpo \cite{AT}. We also construct a version of the algebraic loop map in families and deduce from it the key constancy principle for connected algebraic families.

In Section \ref{s:veryaffine} we recall the necessary background on tropical geometry and Gröbner degenerations. We introduce valuation strata, equip them with natural ind-scheme structures, and prove structural results for loop spaces of very affine varieties. In particular, for schön very affine varieties we compute the connected components of the loop space in terms of initial varieties.

Section \ref{s:alg-loops-very-affine} contains the tropical classification theorem for algebraic loops in schön very affine varieties, Theorem \ref{thm:trop-classification}. We also give an explicit representative in the central fiber of a Gröbner degeneration and show how it computes the algebraic loop after transport to a nearby generic fiber. We conclude the section with a brief discussion of the relation to tropical compactifications, and give another way to compute the loop in terms of rays of a cone (for the polyhedral structure coming from an SNC compactification) associated to the valuation vector of a formal loop. 

In Section \ref{s:hyparr} we specialize to essential central hyperplane arrangements. We describe valuation strata and initial varieties in terms of graded arrangement complements, recall intersection subgroups, and introduce full twists and relative full twists. This leads to an explicit formula for the algebraic loop attached to a formal loop as a product of relative full twists associated with a nested sequence of intersection subgroups.

In the final section, Section \ref{s:cohom-braid-var}, we apply these results to braid varieties. After recalling the categorical input behind the trace property of the character functor coming from the horocycle transform, we prove that for a root hyperplane arrangement, and also for its quotient by the Weyl group, the resulting algebraic loop depends only on the root valuation datum of the formal loop. We then deduce, by an abstract trace argument, that the cohomology of the corresponding braid varieties depends only on root valuation data.

Appendix \ref{app:local-triviality} contains a differential-topological local triviality result for submersions near a compact subset of a fiber. The result is likely standard, but we were unable to find a reference in the precise form needed here.

\subsection{Notation}
Throughout the paper, we denote by $F=\CC(\!(t)\!)$ the field of formal Laurent series, and by $\cO = \CC[[t]]$ its ring of integers. The formal disk is $\cD = \Spec \cO$ with generic point $\eta$ and geometric generic point $\ov{\eta}$, and the formal punctured disk is $\cD^{\times} = \Spec F$. We fix a topological generator $\zeta$ of $\Gal(\ov{F} / F) = \pi_{1}^{\et}(\cD^{\times}, \ov{\eta})$.

\subsection*{Acknowledgements} KJ thanks Nathan Tiggemann and Andreas Gross for answering questions about tropical geometry and hyperplane arrangements, and for pointing to useful references.
MT thanks Chengze Duan for his contributions to an earlier version of this project, and Ch\^au Ng\^o for suggesting the problem of classifying algebraic braids to him during graduate school.

MK and IL were supported by an Australian Research Council Discovery Grant.
MT was supported at early stages of this work by an NSF Graduate Research Fellowship (award DGE-1746045) and an NSF Mathematical Sciences Research Fellowship (award DMS-2002238).
KJ acknowledges support by the Deutsche Forschungsgemeinschaft (DFG, German Research Foundation) Individual Research Grant 566801746 and (through Timo Richarz) by the European Research Council (ERC) under Horizon Europe (grant agreement no 101040935), by the Deutsche Forschungsgemeinschaft (DFG, German Research Foundation) TRR 326 \textit{Geometry and Arithmetic of Uniformized Structures}, project number 444845124 and the LOEWE professorship in Algebra, project number LOEWE/4b//519/05/01.002(0004)/87.

\subsection*{AI Disclosure}
The mathematical results and proofs presented in this paper were
obtained 
by the authors without the use
of generative artificial intelligence. At a later stage of preparing
the manuscript in 2026, we used ChatGPT (OpenAI), 
the GPT-5.6 model, as an editorial and research
aid. Its use included suggesting improvements to the organization,
wording, and clarity of the exposition; helping to refine the
presentation of some arguments; identifying possible ambiguities; and
assisting with literature searches (in particular flagging the key references \cite{AT} and \cite{KU}). All mathematical statements,
proofs, and bibliographic references were independently checked by the
authors, and suggestions produced by ChatGPT were incorporated only
after such verification. The authors take full responsibility for the
content, accuracy, and presentation of the paper.

\section{Algebraic loops}\label{s:algloop}

In this section, we construct the algebraic loop associated with a
formal loop in a smooth complex variety. We begin by recalling the
notions of topological, étale, and formal loops and the natural maps
relating them. Our main result associates with every formal loop
$f:\cD^\times\to Y$ a canonical free homotopy class
$\gamma_f$ in $Y^{\an}$ whose image in the étale fundamental group is
the étale monodromy of $f$. The construction proceeds by extending $f$
to an arc in a smooth compactification of $Y$, approximating this arc
by a convergent one, and restricting the resulting analytic map to a
small circle. A local meridian formula expresses $\gamma_f$ in terms
of the contact orders with the boundary and shows that it is independent
of all choices. We conclude with examples illustrating the resulting
algebraic loop map
\[
\cL_{\alg}:LY(\CC) \longrightarrow\pi_1(Y^{\an})/\!\sim.
\]

\subsection{Recollections on loops} We recall the definitions of three different kinds of loops: topological, étale, and formal, and discuss the relationships between them.

\subsubsection{Topological loops} Let $X$ be a path-connected topological space. By a \emph{topological loop} in $X$ we mean a free homotopy class of continuous maps $S^1 \to X$. We denote by $[S^1, X]$ the set of topological loops in $X$. Then we have a canonical bijection 
\[ 
[S^1, X ] \cong \pi_{1}(X, x) / \sim 
\]
For any two base-points $x$ and $x'$ in $X$, the sets of conjugacy classes $\pi_{1}(X,x) / \sim$ and $\pi_{1}(X,x') / \sim$ are canonically identified. We therefore simply write $\pi_{1}(X)/ \sim$.

\subsubsection{\'Etale loops} 
Let $X$ be a connected scheme and $x$ a geometric point. By an \emph{étale loop} in $X$ we mean a conjugacy class in $\pi_{1}^{\et}(X, x)$. As above, we can drop the base point and denote the set of étale loops by $\pi_1^{\et}(X)/\sim$.

\subsubsection{Topological vs. \'etale loops} 
Now, suppose $X$ is a connected complex variety with analytification $X^\an$. Then we have a canonical map from topological loops to \'etale loops:
\[ 
\varphi: \pi_{1}(X^{\an})/\sim \,\, \to\,\, \widehat{\pi_{1}(X^{\an})}/\sim \,\, \cong \,\, \pi_{1}^{\et}(X)/\sim, 
\]
The injectivity of this map for some cases of interest, e.g. the braid group $B_n$, for  $n\geq4$, is an open problem. This is the question of conjugacy separability.

\subsubsection{Formal loops} Let $X$ be a connected scheme. 
By a \emph{formal loop} in $X$ over a $\CC$-algebra $R$ we mean an $R(\!(t)\!)$-point of $X$. We denote the resulting functor by
\[
LX:\CC\textup{-}\mathrm{Alg}\longrightarrow \mathrm{Set},
\qquad
R\longmapsto X(R(\!(t)\!)).
\]
In particular,
\[
LX(\CC)=X(\CC(\!(t)\!))=X(F)
\]
is the set of formal loops in $X$. If $X$ is affine, then $LX$ has a natural ind-scheme structure.

\subsubsection{Formal vs. \'etale loops} We have a canonical map from formal loops to \'etale loops
\[
LX(\CC) \ra \pi_1^{\et}(X)/\sim,\qquad f\mapsto \ell_f
\]
defined as follows. 
By functoriality, we have a homomorphism  $f_* : \pi_{1}^{\et}(\cD^{\times}, \ov{\eta}) \to \pi_{1}^{\et}(X, f(\bar\eta))$. 
The image of $f_*(\z)$ is well-defined up to conjugacy and gives the étale loop $\ell_f$. 
This construction appeared before in the thesis of the last-named author \cite{Trinh}.

\subsection{Main result} In the previous subsection, we discussed that formal and topological loops map into \'etale loops. Our main result establishes a direct relationship between formal loops and topological loops:

\begin{theorem}\label{thm:Lalg}  Let $Y$ be a smooth connected complex variety. 
Then there exists a canonical map
\[ 
\cL_\alg : LY(\CC) \ra   \pi_{1}(Y^{\an})/\sim,\qquad f\mapsto \gamma_f
\]
such that the following diagram commutes
\[
\xymatrix{
LY(\CC) \ar[r]^-{\cL_\alg} \ar[dr]_-{f\mapsto\ell_f}
    & \pi_1(Y^{\an})/\sim \ar[d]^\varphi \\
    & \pi_1^{\et}(Y)/\sim. 
}
\]
\end{theorem}

\begin{defn} An \emph{algebraic loop} is an element in the image of $\cL_\alg$.  The set of algebraic loops is denoted by $[S^1,Y^{\an}]^{\alg}$. 
\end{defn}

\begin{remark} We can think of $\cL_\alg$ as recording the topological monodromy of a
formal loop. Understanding the fibers of this map (e.g. their motivic geometry) is an interesting open problem. 
If $Y$ is affine, then $\cL_\alg$ should induce a natural surjection
\[
\pi_0(LY)\twoheadrightarrow [S^1,Y^\an]^\alg,
\]
where $\pi_0$ denotes the set of connected components of the
ind-scheme. In the next section, we prove that this map is a bijection when
$Y$ is very affine and describe its image using tropical geometry.
\end{remark}

\subsubsection{Outline of the proof} 
The proof of the theorem is given in the next subsection. 
It proceeds by choosing a smooth compactification $X$ of $Y$ with SNC boundary divisor $D$ and extending $f$ to an arc
$g:\cD\to X$. Let  $x=g(0)$ be the arc center, $D_1,\ldots,D_r$ the
irreducible components of $D$ passing through $x$, and
$m_i$  the contact order of $g$ with $D_i$. Choose a sufficiently small analytic neighborhood
$U\subset X^{\an}$ of $x$ such that
\[
U^\circ:=U\cap Y^{\an}\cong
(\Delta^\times)^r\times\Delta^{n-r}.
\]
Let
$\iota:U^\circ\hookrightarrow Y^{\an}$ denote the inclusion and $\mu_i$ the positive meridian around $\{z_i=0\}$. Then the loop constructed in the theorem is
\begin{equation} \label{eq:meridian} 
\gamma_f=
\left[
\iota_*\bigl(\mu_1^{m_1}\cdots\mu_r^{m_r}\bigr)
\right]
\in\pi_1(Y^{\an})/\sim.
\end{equation} 
Thus $\gamma_f$ is obtained by recording the contact orders of $g$ with
the local boundary components and mapping the resulting loop in $U^\circ$ into
$Y^{\an}$.  Although
this description uses a compactification and local coordinates, the proof below shows that the resulting conjugacy class is independent of all these choices.

\begin{remark}\label{AT}
The existence of the map in Theorem \ref{thm:Lalg} can also be deduced from the Betti realization functor $\Psi$ constructed in the work of Achinger and Talpo \cite{AT}. Indeed, if $Y$ is smooth and $Y_F=Y\times_{\bC} F$, then
\cite[Theorem 1.1]{AT} gives natural identifications
\[
\Psi(\Spec F)\simeq S^1,
\qquad
\Psi(Y_F)\simeq Y^{\an}\times S^1
\]
over $S^1$. Consequently, a formal loop $f:\Spec F\to Y_F$ induces a section $\Psi(f):S^1\longrightarrow Y^{\an}\times S^1$, and hence a free homotopy class in $Y^{\an}$. One easily checks that the resulting loop agrees with our construction. In fact, their result gives an algebraic loop map for any finite type complex scheme, not necessarily smooth. Moreover, using the approach of \cite{AT}, one can extend the map $\cL_\alg$ to the setting of Artin stacks by smooth descent. 

Their comparison with the étale homotopy type \cite[Theorem 4.1]{AT} also gives the compatibility with the étale loop appearing in Theorem \ref{thm:Lalg}. We nevertheless include our proof because it is framed in more elementary terms, and gives an explicit realization of the resulting loop in terms of boundary meridians. 
\end{remark}

The functorial point of view of \cite{AT} gives the following naturality for the algebraic loop map.

\begin{cor}\label{cor:naturality} Let $f : Y \to Z$ be a morphism of finite type smooth $\CC$-schemes. Then the following diagram commutes
\[
\xymatrix{ 
LY(\CC) \ar[r]^{\cL_{\alg,Y}} \ar[d]^{f} & [S^1,Y^{\an}] \ar[d]^{f_*} \\
LZ(\CC) \ar[r]^{\cL_{\alg,Z}} & [S^1,Z^{\an} ],
}
\]
where $f_*([\g]) = [ f \circ \g]$.
\end{cor}

\subsubsection{Examples} \label{expl:algloops}
We now give some examples of this construction: 

\begin{enumerate} 
\item If $f$ extends to an arc in $Y$, then $\gamma_f$ is trivial. In particular,
$\cL_\alg(Y(\cD))=\{1\}$ 
and if $Y$ is proper, then we also have $\cL_\alg(Y(\cD^\times))=\{1\}$. 
\item If $Y=\mathbb G_m$, then 
$\cL_{\alg}:\bC(\!(t)\!)^\times\longrightarrow\mathbb Z$ is given by   
$f\longmapsto \val_t(f)$. 
\item If $Y=T$ is a torus, then  $\cL_{\alg}:T(\bC(\!(t)\!)) \longrightarrow \pi_1(T)= X_*(T)$ is given by  
\[
f\longmapsto \Big(\chi \mapsto \val_t(\chi(f))\Big),\qquad \chi \in X^*(T). 
\]
\item If $Y=G$ is a connected reductive group with coroot lattice $Q^\vee$, then 
\[
\cL_\alg: G(\cD^\times) \ra \pi_1(G)=X_*(T)/Q^\vee
\]
 is the Kottwitz homomorphism \cite{Kottwitz}.  By (1), this map factors through the affine Grassmannian $L G/ L^+ G$. 
\item Suppose $Y=\bP^1-\{0,1,\infty\}$ and write 
\[
\pi_1(Y)=\langle a, b, c\, : \, abc=1\rangle  \cong F_2. 
\]
Then
\[
\cL_\alg(Y(\cD^\times)) = \{ [a^n], [b^n], [c^n] : n\in \mathbb{Z}_{\geq 0}\}. 
\]

\item If $Y$ is a complement of a finite central hyperplane arrangement in $\bC^n$, then a precise formula for $\cL_\alg$ is given in
\S \ref{s:hyparr}. 

\item Suppose $Y$ is the configuration space of $n$ distinct unordered points
in $\bC$, so that $\pi_1(Y^\an)$ is isomorphic to the classical braid group $B_n$. Then local braid monodromies
of plane curve singularities define algebraic loops in $Y$, and their
closures are algebraic links. We will explore this connection and generalisations to arbitrary finite reflection groups 
in a follow-up paper.

\item The theorem has an orbifold version for smooth Deligne--Mumford
stacks. Let $Y=\cM_{g,n}$ be the moduli stack of genus $g$ curves with $n$ marked points. Then the orbifold fundamental group of $Y$ is the mapping class group $\Gamma_{g,n}$. We expect that algebraic loops in $Y$ are closely related to pseudo-periodic mapping classes \cite{MM}. 

\item Let $Y=\cA_g$ denote the moduli stack of principally polarized abelian varieties of dimension $g$. Then $\pi_1^{\mathrm{orb}}(\cA_g^\an)\cong\Sp_{2g}(\mathbb Z)$. We expect that algebraic loops in $Y$ are quasi-unipotent conjugacy classes in $\Sp_{2g}(\mathbb Z)$ (this should follow from Grothendieck's monodromy theorem). It would be interesting to characterize precisely which such conjugacy
classes arise.
\end{enumerate}

\subsection{Proof of Theorem \ref{thm:Lalg}}  
The construction of $\gamma_f$ and the proof of its properties proceed in
six steps:

\begin{itemize}
\item Choose a smooth proper compactification $Y\hookrightarrow X$ whose
boundary is a simple normal crossing divisor, and extend $f$ uniquely to a
formal arc $g\colon\cD\to X$.

\item Apply Artin approximation to approximate $g$ to sufficiently high order
by an algebraic arc $\widetilde g$. Since algebraic power series are
convergent, $\widetilde g$ determines a holomorphic map
$\widetilde g\colon\Delta_\varepsilon\longrightarrow X^{\an}$. 
for some $\varepsilon>0$.

\item Show that, after shrinking $\varepsilon$,
$\widetilde g(\Delta_\varepsilon^\times)\subset Y^{\an}$. 
Restricting $\widetilde g$ to a sufficiently small circle then gives a
candidate for $\gamma_f$.

\item Establish formula \eqref{eq:meridian}, which expresses the resulting
loop as a product of meridians about the local boundary components, with
exponents given by the contact orders of $g$. This proves independence of
the approximation, the approximation order, and the chosen radius.

\item Use the meridian formula, together with the comparison between
topological meridians and étale inertia, to show that the image of
$\gamma_f$ in $\pi_1^{\et}(Y)/\sim$ is $\ell_f$.

\item Finally, compare any two SNC compactifications through a common
resolution to prove that $\gamma_f$ is independent of the chosen
compactification.
\end{itemize}

\subsubsection{Compactification} By Nagata compactification and Hironaka resolution of singularities, we can find a smooth proper compactification $Y\hookrightarrow X$ with SNC boundary divisor $D$.  
By the valuative criterion for properness, the formal loop
$f\colon \cD^{\times}\to Y$ extends uniquely to an arc
\[
g\colon \cD\to X
\]
 We let  $x:=g(0)$ denote the center of this arc. Next, we consider the contact orders of $g$ with the components of $D$. Write
\[
D=\bigcup_{i=1}^{s}D_i
\]
for the decomposition of the boundary into its irreducible components, and let
$\cI_i\subseteq\cO_X$ be the ideal sheaf of $D_i$. The pullback ideal
$g^{-1}\cI_i\cdot\cO_{\cD}\subseteq\cO_{\cD}=\bC[\![t]\!]$
is nonzero because $g(\cD)\nsubseteq D_i$. (Indeed, $g|_{\cD^\times}=f$.) As
$\bC[\![t]\!]$ is a PID, this ideal has the form
$(x_i(t))$ for some nonzero power series $x_i(t)$, unique up to multiplication
by a unit. The contact order of $g$ with $D_i$ is defined by
\[
m_i:=\ord_t(x_i(t))\in\mathbb Z_{\geq 0}.
\]
Set
\[
N_0:=\max\{0,m_1,\ldots,m_s\}.
\]
Notice that $m_i=0$ whenever $g(0)\notin D_i$. In particular, if $f$ extends
to an arc in $Y$ (e.g. when $Y$ is proper), then $m_i=0$ for every $i$ and hence $N_0=0$.

\subsubsection{Artin approximation} Let $\CC\langle t\rangle\subset\CC[\![t]\!]$ 
be the ring of algebraic power series, i.e. formal power series algebraic over
$\CC(t)$. Fix an integer $N>N_0$. By Artin approximation (\cite[Theorem 1.10]{Artin}), we may choose
$\widetilde g\in X(\Ctalg)$ such that
\[
\widetilde g\bmod t^N=g\bmod t^N.
\]
In particular, $\widetilde g(0)=g(0)=x$. By the convergent Newton--Puiseux theorem (cf. \cite{Wall}, \S 2), every algebraic power series is
convergent. 
Consequently, $\widetilde g$ defines a holomorphic map
\[
\widetilde g\colon\Delta_\varepsilon\longrightarrow X^{\an}
\]
for some $\varepsilon>0$.

\subsubsection{Defining $\gamma_f$} First,  
we claim that, after shrinking $\varepsilon$ if necessary,
$\widetilde g\bigl(\Delta_\varepsilon^\times\bigr)\subset Y^{\an}$. 
After reindexing, suppose that $D_1,\ldots,D_r$ are precisely the
irreducible components of $D$ containing $x$. The SNC condition means that there
exists an analytic neighborhood $U\subset X^{\an}$ of $x$, with holomorphic
coordinates $z_1,\ldots,z_n$, such that
\[
D\cap U=\bigcup_{i=1}^{r} (D_i\cap U)= \bigcup_{i=1}^r\{z_i=0\},
\qquad
U^\circ:=U\cap Y^{\an}\cong(\Delta^\times)^r\times\Delta^{n-r}.
\]
If $x\in Y$, we take $r=0$ and choose $U\subset Y^{\an}$. After shrinking $\varepsilon$, we may assume that
$\widetilde g(\Delta_\varepsilon)\subset U$. For $1\leq i\leq r$, set
\[
a_i(t):=z_i(\widetilde g(t)).
\]
Since $g^*z_i=t^{m_i}u_i(t)$ with $u_i(t)\in\CC[\![t]\!]^\times$, and since
$\widetilde g\equiv g\pmod{t^N}$ with $N>N_0\geq m_i$, we have
$\ord_t(a_i)=m_i$. 
Thus
$a_i(t)=t^{m_i}v_i(t)$
for some convergent power series (i.e. holomorphic function) $v_i$ with $v_i(0)\neq 0$. Since zeros of a non-constant holomorphic function are isolated, we may shrink
$\varepsilon$ once more to ensure that $v_i$ is nonvanishing on
$\Delta_\varepsilon$. Therefore
$\widetilde g(\Delta_\varepsilon^\times)\subset U^\circ$,
establishing the claim. 

For any $0<\rho<\varepsilon$, set 
\[
\gamma_{\widetilde g,\rho}\colon S^1\to U^\circ,
\qquad
e^{\mathrm{i}\theta}\longmapsto
\widetilde g(\rho e^{\mathrm{i}\theta}).
\]
Let
$\iota:U^\circ\hookrightarrow Y^{\an}$ be the inclusion map.
 
 \begin{defn} We define $\gamma_{f}$ to be the free homotopy class of 
$\iota \circ \gamma_{\widetilde g,\rho}: S^1\ra Y$. 
\end{defn}

\subsubsection{Meridian formula} Note that 
\[
\pi_1(U^\circ) \cong  \pi_1((\Delta^\times)^r\times\Delta^{n-r}) \cong \ZZ^r. 
\]
More precisely, 
choose a basepoint $u \in U^\circ$ and write $u=(a_1,...,a_r,b)$ with $a_i\in \Delta^\times$ and $b\in \Delta^{n-r}$. The positive meridian around $\{z_i=0\}$ is defined by 
\[
\mu_i(s)=(a_1,\cdots, a_{i-1}, a_i e^{2\pi i s}, a_{i+1}, \cdots, a_r, b). 
\]
Then $\pi_1(U^\circ, u)$ is the free abelian group generated by $\mu_i$'s. We claim that there is a free homotopy 
\[
\gamma_{\widetilde g,\rho}\sim \mu_1^{m_1}\cdots \mu_{r}^{m_r}. 
\]
Indeed, 
$z_i(\widetilde g(t))=t^{m_i}v_i(t)$, 
where $v_i$ is nonvanishing on $\Delta_\varepsilon$. Thus
$t^{m_i}$ contributes winding number $m_i$, while $v_i$ contributes winding
number zero. This proves the claim and shows  that  
$\gamma_f$  equals the conjugacy class of $\iota_*(\mu_1^{m_1}\cdots \mu_{r}^{m_r})$. 
In particular, $\gamma_{f}$ is independent of $\rho$, of $N$, and of the chosen
approximation $\widetilde g$.

\subsubsection{Compatibility with the \'etale loop}

Let
\[
V:=\Spec\bigl(\cO_{X,x}^{\mathrm{sh}}\bigr),
\qquad
u:V\to X,
\qquad
V^\circ:=Y\times_XV.
\]
We have 
$\pi_1^\et(V^\circ)\cong\widehat{\ZZ}^{\,r}$, 
where the standard finite quotients are given by the simultaneous Kummer
covers $w_i^q=z_i$. Since $\bC[\![t]\!]$ is strictly henselian, the arc $g:\cD\to X$ lifts
uniquely to an arc $g_V:\cD\to V$. Restricting to $\cD^\times$ gives a
factorization
\[
\xymatrix{
\cD^\times \ar[r]^-p \ar[dr]_-f
    & V^\circ \ar[d] \\
    & Y.
}
\]
Since
\[
g^*z_i=t^{m_i}u_i(t),
\qquad
u_i(t)\in\bC[\![t]\!]^\times,
\]
pulling back $w_i^q=z_i$ along $p$ gives
\[
w_i^q=t^{m_i}u_i(t).
\]
The unit $u_i(t)$ has a $q$-th root, so the image of the inertia $\z$ in
$(\ZZ/q)^r$ is $(m_1,\ldots,m_r)$ modulo $q$. Hence
\[
p_*(\z)=(m_1,\ldots,m_r)
\in\widehat{\ZZ}^{\,r}.
\]
Under the comparison isomorphism, the standard generators correspond to the
positive meridians $\mu_1,\ldots,\mu_r$. The meridian formula therefore gives
$\varphi(\gamma_f)=[f_*(\z)]=\ell_f$.

\subsubsection{Independence of compactification} 
 Let $X_1$ and
$X_2$ be two smooth proper compactifications of $Y$ with SNC boundary
divisors. By resolving the closure of the graph of the birational map
$X_1\dashrightarrow X_2$, we may choose a third such compactification
$X_3$ and morphisms
\[
\pi_j\colon X_3\to X_j,\qquad j=1,2,
\]
restricting to the identity on $Y$. Let $g_j\colon\cD\to X_j$ and $g_3\colon\cD\to X_3$ be the corresponding
extensions of $f$. Choose $N$ larger than the relevant contact orders for
all three compactifications, and choose an order-$N$ algebraic approximation
$\widetilde g_3\in X_3(\Ctalg)$
to $g_3$. Then $\pi_j\circ\widetilde g_3$ is an order-$N$ approximation to
$g_j$. Moreover, on the punctured disk the maps $\pi_j$ are the identity on
$Y^{\an}$. Therefore the loop constructed from $\widetilde g_3$ agrees with
the loop constructed from $\pi_j\circ\widetilde g_3$ in $Y^{\an}$. By the independence already proved for a fixed compactification, it follows
that
\[
\gamma_{f,X_1}=\gamma_{f,X_3}=\gamma_{f,X_2}.
\]
Thus, $\gamma_f$ is independent of $X$. 
\qed

\subsection{Algebraic loops in families}
Using the machinery in Achinger--Talpo \cite{AT}, we can construct a version of the algebraic loop map in families for a smooth complex variety $Y$. Let $\cS$ be the $\infty$-category of spaces, and let $\Sch_{\CC}^{\mathrm{lft}}$ be the category of complex schemes which are locally of finite type. For a complex variety $Y$, we consider the loop functor $LY$ as presheaf $(\Sch_{\CC}^{\mathrm{lft}})^{\mathrm{op}} \to \cS$ (which is set-valued), and we consider the presheaf
\[L^{\mathrm{top}}Y :  (\Sch_{\CC}^{\mathrm{lft}})^{\mathrm{op}} \to \cS, \qquad S \mapsto \Map(S^{\an}\times S^1,Y^{\an}), \]
which is the topological loop space for $Y^{\an}$.

\begin{theorem}\label{thm:alg-loop-families}
Let $Y$ be a smooth complex scheme locally of finite type. There is a natural
transformation of presheaves
\[
\cL_{\alg,Y} :LY\longrightarrow L^{\mathrm{top}}Y
\]
on $\Sch_{\CC}^{\mathrm{lft}}$, inducing the algebraic loop map when $S=\Spec \CC$.
\end{theorem}

In other words, for every affine complex scheme
$S=\Spec A$ and every morphism
\[
f:\Spec A((t))\longrightarrow Y,
\]
there is a functorially associated point
\[
\cL_{\alg,Y}(f)\in
\Map(S^{\an}\times S^1,Y^{\an}).
\]
Note that the above natural transformation gives a direct comparison between the algebraic loop space and the topological loop space.  For the proof, we need the following lemma.

\begin{lemma}\label{lem:h-hypersheaf}
Let $Y$ be a complex scheme locally of finite type. The presheaf $L^{\mathrm{top}}Y$ is a hypersheaf for the $h$-topology.
\end{lemma}

\begin{proof}
Consider the functor
\[
(-)^{\an} :\Sch_{\mathbb C}^{\mathrm{lft}} \longrightarrow \mathcal S
\]
associating to $X$ the space $X^{\an}$. This functor satisfies $h$-hypercodescent, in other words, for every $h$-hypercover
$U_\bullet\to S$, the canonical map
\[
\operatorname*{colim}_{[n]\in\Delta^{\mathrm{op}}}
U_n^{\an}
\longrightarrow
S^{\an}
\]
is an equivalence of spaces (here $\Delta$ denotes the simplex category) -- this is the special case of \cite[Proposition 2.9]{AT} applied to the constant family. Taking the product with $S^1$ preserves colimits, so 
\[
S^{\an}\times S^1
\simeq
\operatorname*{colim}_{[n]\in\Delta^{\mathrm{op}}}
(U_n^{\an}\times S^1).
\]
The functor $\Map(-,X) : \cS^{\mathrm{op}} \to \cS$ preserves limits, hence it takes colimits in the first variable to limits. Therefore
\[
\begin{aligned}
L^{\mathrm{top}}Y(S)
&=
\Map(S^{\an}\times S^1,Y^{\an})\\
&\simeq
\Map\!\left(
\operatorname*{colim}_{[n]}
(U_n^{\an}\times S^1),
Y^{\an}
\right)\\
&\simeq
\operatorname*{lim}_{[n]\in\Delta}
\Map(U_n^{\an}\times S^1,Y^{\an})\\
&=
\operatorname*{lim}_{[n]\in\Delta}
L^{\mathrm{top}}Y(U_n),
\end{aligned}
\]
proving the claim.
\end{proof}

\begin{proof}[Proof of Theorem \ref{thm:alg-loop-families}]
Recall that $F=\mathbb C((t))$ and $\cO=\mathbb C[[t]]$. We first construct $\cL_{\alg,Y}$ on smooth finite type test schemes, and then extend by $h$-descent. So let $S=\Spec A$ be smooth over $\mathbb C$. Consider
the $t$-adic formal scheme
\[
\mathfrak S=\Spf \, A[[t]]
\]
over $\Spf\, \cO$. Its rigid generic fiber is the affinoid rigid
$F$-space
\[
\mathfrak S_\eta=\Spa(A((t)),A[[t]])
\]
Indeed, if  $A=\mathbb C[x_1,\ldots,x_n]/I$, 
then $ A[[t]] \simeq \cO\langle x_1,\ldots,x_n\rangle/I$, and moreover \[
A((t))
\simeq
F\langle x_1,\ldots,x_n\rangle/I.
\]
Since $S$ is smooth, $\mathfrak S$ is a smooth formal
$\cO$-scheme and $\mathfrak S_\eta$ is a smooth,
quasi-compact and separated rigid $F$-space.

The morphism $f$ induces a morphism of rigid spaces
\[
f^{\rig}:
\mathfrak S_\eta
\longrightarrow
(Y_F)^{\an}_{\rig},
\qquad
Y_F=Y\times_{\mathbb C}F.
\]
We can therefore apply the rigid Betti realization functor from \cite[Theorem 1.2]{AT}, to obtain the map
\[
\Psi_{\rig}(f^{\rig}):
\Psi_{\rig}(\mathfrak S_\eta)
\longrightarrow
\Psi_{\rig}\bigl((Y_F)^{\an}_{\rig}\bigr).
\]

The good-reduction property (3) in \cite[Theorem 1.2]{AT} gives a natural equivalence
\[
\Psi_{\rig}(\mathfrak S_\eta)
\simeq
S^{\an}\times S^1.
\]
By \cite[Corollary~5.23]{AT} there is a natural equivalence
\[
\Psi_{\rig}\bigl((Y_F)^{\an}_{\rig}\bigr)
\simeq
\Psi(Y_F)\]
in $\cS_{/S^1}$. Moreover,  \cite[Corollary~2.11]{AT} implies
\[
\Psi(Y_F) \simeq Y^{\an}\times S^1,
\]
and we obtain the map
\[
\Psi_{\rig}(f^{\rig}) : S^{\an}\times S^1
\longrightarrow
Y^{\an}\times S^1.
\]
Taking its first component defines
\[
\cL_{\alg,Y}(f):
S^{\an}\times S^1
\longrightarrow
Y^{\an}.
\]

For naturality, note that a morphism $S'=\Spec A'\to S=\Spec A$ induces a morphism of formal schemes
\[
\Spf A'[[t]]
\longrightarrow
\Spf A[[t]]
\]
and hence a morphism of their rigid generic fibers. Functoriality of
$\Psi_{\rig}$ and naturality of the good-reduction comparison give
base-change compatibility. 

Now note that any finite-type complex scheme $S$ admits a smooth $h$-hypercover by resolution of singularities. The above construction therefore induces a natural transformation of the $h$-hypersheafifications of $LY$ and $L^{\mathrm{top}}Y$. By Lemma \ref{lem:h-hypersheaf}, $L^{\mathrm{top}}Y$ is an $h$-hypersheaf, so we obtain the natural transformation $\cL_{\alg,Y}$ by composing this induced transformation with the natural one from $LY$ to its hypersheafification. Finally, the construction extends to locally finite type schemes by Zariski descent.

When $S=\Spec\mathbb C$, the source is $\Spa(F,\cO)$. The comparison
\[
\Psi_{\rig}\bigl((Y_F)^{\an}_{\rig}\bigr)\simeq\Psi(Y_F)
\]
identifies the resulting section over $S^1$ with the original algebraic loop. Hence $\cL_{\alg,Y}$ induces
$\cL_{\mathrm{alg}}$ on $\mathbb C$-points.
\end{proof}

In the proof we use the rigid analytic Betti realization $\Psi^{\rig}$ of \cite{AT}. It is not known whether it satisfies $h$-descent. If it does, the above construction can be extended to finite type $\CC$-schemes which are not necessarily smooth.
The construction is compatible with base change. For $S=\Spec\mathbb C$, this construction agrees with the algebraic loop construction of Achinger--Talpo, and hence with the explicit construction of Theorem~\ref{thm:Lalg}.

The construction of the family version of the algebraic loop map has the following consequence for constancy of $\cL_{\alg}$. 

\begin{cor}\label{cor:constancy-components} Let $Y$ be a smooth complex scheme locally of finite type, and let $S$ be a connected complex scheme, locally of finite type. Let $f \in LY(S)$, and for a complex point $s \in S(\CC)$ let $f_s := f \circ s \in LY(\CC)$. Then the map
\[ S(\CC) \to [S^1,Y^{\an}], \qquad s \mapsto \cL_{\alg}(f_s) \] 
is constant.
\end{cor}

\begin{proof}
The $S$-point $f \in LY(S)$ is sent to 
\[\cL_{\alg,Y}(f) \in \Map(S^{\an} \times S^1, Y^{\an}) \cong \Map(S^{\an}, \Map(S^1, Y^{\an}), \]
so we get 
\[
\cL_{\alg,Y}(f) : S^{\an}\longrightarrow \Map(S^1,Y^{\an}).
\]
Since $S$ is connected, its analytification $S^{\an}$ is connected.
Consequently, the image of this map is contained in a single connected
component of $\Map(S^1,Y^{\an})$. As
\[
\pi_0(\Map(S^1,Y^{\an}))=[S^1,Y^{\an}],
\]
the induced map $S(\mathbb C)\longrightarrow [S^1,Y^{\an}]$ is constant.
\end{proof}
In general, even when $Y$ is affine, it is not clear that any two points in the same connected component of $LY$ can be joined by a finite chain of connected schemes locally of finite type. However, we will later specialize to very affine varieties which satisfy a standard smoothness assumption (schön), and see that in this case, the algebraic loop map is constant on connected components.

\section{Geometry of tropical fibers}\label{s:veryaffine}
This section develops the tropical framework used to study formal loops in very affine varieties. We first recall the intrinsic torus of a very affine variety, the valuation point associated with a formal loop, and the resulting tropicalization. We then review the Gröbner degeneration $\cY_w\to\AA^1$ attached to an integral tropical point $w$, whose special fiber is the initial variety $\init_w(Y)$. Using these degenerations, we decompose the formal loop space $LY$ into valuation strata $L_wY$ and identify each stratum with the infinite Greenberg scheme of the corresponding formal Gröbner degeneration. Reduction modulo $t$ defines an initial-term morphism
\[
L_wY\longrightarrow\init_w(Y),
\]
which records the leading term of a formal loop. When $Y$ is schön, these morphisms determine the connected components of $LY$, yielding a natural identification between $\pi_0(LY)$ and the enhanced integral tropicalization of $Y$.

\subsection{Very affine varieties and tropicalization} In this subsection, we recall how the intrinsic torus of a very affine variety allows us to assign tropical valuation points to formal loops. 

\begin{defn} A complex variety $Y$ is called \emph{very affine} if it is affine and its coordinate ring $\cO(Y)$ is generated by units.
\end{defn}
Any very affine variety can be embedded into a complex torus $T$. We briefly recall how this embedding is constructed. By a theorem of Samuel, the quotient group $\cO(Y)^{\times} / \CC^{\times}$ is a finitely generated free abelian group. The associated complex torus $T$ with complex points
\[ T(\CC) = \Hom(\cO(Y)^{\times}/\CC^{\times}, \CC^{\times}) \]
is called \emph{intrinsic torus} of $Y$. 
Choosing a basis of $\cO(Y)^{\times}/\CC^{\times}$ and representatives
$u_1,\ldots,u_n\in\cO(Y)^\times$ of its elements determines a closed
embedding
\[
Y\hookrightarrow (\Gm)^n,\qquad
y\longmapsto (u_1(y),\ldots,u_n(y)).
\]
Conversely, any closed subvariety of a complex torus $(\Gm)^n$ is automatically very affine. Its intrinsic torus may differ from $(\Gm)^n$.

Now, let $f :\cD^{\times} \to Y$ be a formal loop in $Y$, with associated pullback map $f^* :  \cO(Y) \to F$. Then $f$ determines a valuation function
\begin{align*} \val_{f} : \cO(Y)^{\times} / \CC^{\times} &\to \ZZ  \\
u \mapsto \val_{t}(f^*u). 
\end{align*}
Since the intrinsic torus $T$ of $Y$ has $\CC$-points $\Hom(\cO(Y)^{\times}/\CC^{\times}, \CC^{\times})$, the function $\val_{f}$ is the same as a cocharacter of $T$, i.e.
\[ \val_{f} \in \Hom(\cO(Y)^{\times}/\CC^{\times}, \ZZ) = X_*(T). \]

Let $\ov{F}$ be the field of Puiseux series. Then more generally we may consider the map
\begin{align*}
\val: Y(\ov{F}) &\to X_{*}(T)_{\QQ} = X_*(T) \otimes_{\ZZ} \QQ \\
y &\mapsto \val_{y}.
\end{align*}
The closure in $X_*(T)_{\RR}$ (endowed with the Euclidean topology) of the image of this map is denoted by $\Trop(Y) \subset X_*(T)_{\RR}$, and we call it the \emph{tropicalization} of $Y$. Choosing a $\ZZ$-basis of $X_*(T)$ determines an isomorphism $X_*(T)_{\RR} \cong \RR^n$, and the image of $\Trop(Y)$ in $\RR^n$ is a tropical variety, see \cite[Theorem 3.2.3]{MS} for details. Moreover, we call $\Trop(Y)(\ZZ) = \Trop(Y) \cap X_*(T)$ the \emph{integral points} of the tropicalization. Every formal loop $f : \cD^{\times} \to Y$ canonically determines an integral point $w_f\in\Trop(Y)\cap X_*(T)=\Trop(Y)(\ZZ)$.

\begin{exam}\label{ex:hyparr} Let $V$ be a complex vector space and $\cA$ a finite essential linear hyperplane arrangement in $V$. Here, essential means that the intersection of all hyperplanes is trivial. Then the hyperplane complement $M=M(V,\cA):=
V\setminus\bigcup_{H\in\cA}H$ is very affine. Indeed, writing 
\[\cA = \{H_{1},...,H_{n} \} \]
for hyperplanes $H_i$ we may choose linear functionals $\ell_{i} \in V^*$ such that $H_{i} = \mathrm{ker}(\ell_{i})$. Then we obtain a closed  embedding 
\begin{align*} 
 M 	&\to (\Gm)^n \\
 x 	&\mapsto (\ell_{1}(x),...,\ell_{n}(x)).
 \end{align*}
The torus $\Gm^{n}$ is identified with the intrinsic torus of $M$ in this case, because $(\ell_{1},...,\ell_{n})$ is a $\ZZ$-basis of $\cO(M)^{\times}/\CC^{\times}$. Under this identification, the tropical point associated with a formal
loop $f\in M(F)$ is
\[
w_f=
\bigl(
\val_t(\ell_1(f)),\ldots,\val_t(\ell_n(f))
\bigr)\in\Trop(M)(\ZZ).
\]
\end{exam}

\subsection{Gröbner degenerations and initial varieties}  In this subsection, we recall the Gröbner degeneration associated with a cocharacter
$w$, which degenerates $Y$ to the variety determined by its
$w$-initial terms.

Let $Y$ be a very affine variety with intrinsic torus $T$, and choose an embedding $Y\hookrightarrow T$. Set
$M=X^*(T)$,  $N=X_*(T)$, 
and denote the canonical pairing by
\[
\langle-,-\rangle:M\times N\longrightarrow\ZZ.
\]
Let $I\subset\CC[M]=\cO(T)$ be the ideal defining $Y$ in $T$. To any cocharacter $w\in N=X_*(T)$ one associates a flat morphism $\mathcal Y_w\to \AA^1$, called the \emph{Gröbner degeneration} of $Y$ with respect to $w$, whose restriction over $\Gm$ is isomorphic to $Y\times\Gm$. Its fiber over $0$ is the \emph{initial variety}  $\init_w(Y)$. The special fiber is nonempty—and hence $\mathcal Y_w\to\AA^1$ is surjective—if and only if $w\in\Trop(Y)(\ZZ)$.

The initial variety is defined as follows. For any function 
\[ f = \sum_{m \in M} c_m \chi^{m} \in \CC[M] \]
let $a_{w}(f) = \min\{ \langle m,w \rangle \in \ZZ \mid c_m \neq 0 \}$. The $w$-initial form of $f$ is 
\[ \init_{w}(f) = \sum_{\langle m,w\rangle = a_{w}(f)} c_m\chi^m. \]
In other words, $\init_{w}(f)$ consists only of those monomials in $f$ with minimal $w$-weight. Then, the initial ideal of $I$ is the ideal generated by initial forms, i.e.
\[ \init_{w}(I) = \langle \init_{w}(f) \mid f \in I \rangle, \]
and the initial variety $\init_{w}(Y)$ is the subvariety of $T$ cut out by $\init_{w}(I)$. The point $w$ lies in $\Trop(Y)$ precisely when $\init_w(I)$ contains no monomial. Equivalently, for every $f\in I$, the minimum
\[
a_w(f)=\min\{\langle m,w\rangle:c_m\neq 0\}
\]
is achieved for at least two exponents $m$. Thus, if $w\notin\Trop(Y)$, there exists $f\in I$ such that $\init_w(f)$ is a monomial. Since every monomial is a unit in the Laurent polynomial ring $\CC[M]$, this implies $\init_w(I)=\CC[M]$, and hence $\init_w(Y)=\varnothing$.

The Gröbner degeneration $\cY_w$ is constructed as follows, cf. \cite[\S 15.8]{Eis}. Let $R=\CC[t]$, and consider the algebra 
\[R[M]^{w} = R[t^{-\langle m,w \rangle}\chi^{m} \mid m\in M] \subset R[t^{-1}][M]. \]
We then define
\[\cY_{w} := \Spec( R[M]^{w}  / I_{\CC(t)} \cap R[M]^{w} ), \]
which has a natural map $\cY_{w} \to \AA^1$ induced from the inclusion of $R$ into its coordinate ring. To see that this map is flat, one shows that $\CC[\cY_{w}]$ is torsion-free over $R$ (which suffices because $R$ is a PID). The special fiber is identified with the initial variety
\[ \cY_{w,0} = \init_{w}(Y).\]

This family can alternatively be described as follows. Since $w$ is a cocharacter of $T$, we have a $\Gm$-action on $T$ via $w$, and for every $t\in \Gm$ we get a translate $Y_t = t^{-w}Y$ of $Y$. Let 
\[\cY^{\circ} = \{(t, t^{-w}y) \in \Gm \times T \mid y \in Y \}. \]
Then $\cY_{w}$ is the scheme-theoretic closure of $\cY^{\circ}$ in $\AA^1 \times T$. In this description, it is easy to see that the morphism
\[ Y \times \Gm \to \cY^{\circ},\qquad  (y,t) \mapsto (t,t^{-w}y) \]
is an isomorphism. Thus, the restriction of $\cY_w\to\AA^1$ over $\Gm$ is the trivial family $Y\times\Gm\to\Gm$.

\subsection{The tropical fiber and valuation strata}
\label{ss:valuation-strata}
We continue using the notation introduced above. Thus, $Y$ is a very affine variety with intrinsic torus $T$ and we have a valuation map
\[
\val:Y(F)\longrightarrow
X_*(T)\cap\Trop(Y)=\Trop(Y)(\ZZ).
\]

With the convention of Section~\ref{s:algloop}, the loop space is the functor
\[
LY(R)=Y(R(\!(t)\!)).
\]
Thus, $LY(\CC)=Y(F)$. In this subsection, we study the decomposition of the loop space $LY$ into valuation strata. For each $w\in\Trop(Y)(\ZZ)$, we equip the set-theoretic fiber $\val^{-1}(w)$ with a natural scheme structure and identify it with the infinite Greenberg scheme associated with the Gröbner degeneration $\cY_w$. This description produces an initial-term morphism  
\[
L_wY\ra \init_w(Y)
\]
to the initial variety $\init_w(Y)$. 
When $Y$ is schön this gives a description of the connected components of $LY$ in terms of the connected components of its initial varieties. To state the precise results, we need some definitions.

\begin{defn} For $w\in\Trop(Y)(\ZZ)$,
the \emph{valuation stratum} $L_wY$ is the subfunctor of $LY$ defined, for every $\CC$-algebra $R$, by
\[
L_wY(R)
=
\left\{
f\in Y(R(\!(t)\!))
\;\middle|\;
t^{-w}f\in\cY_w(R[[t]])
\right\}.
\]
Here, an $R[[t]]$-point of $\cY_w$ is understood to be a morphism
\[
\Spec R[[t]]\longrightarrow\cY_w
\]
over $\AA^1$, where $\Spec R[[t]]\to\AA^1$ is induced by the homomorphism
$\CC[t]\to R[[t]]$ sending $t$ to $t$.
\end{defn}

We shall also use the following enhancement of the set of integral tropical points.

\begin{defn}\label{def:enhanced-tropicalization}
The \emph{enhanced integral tropicalization} of $Y$ is the set
\[
\mathcal T(Y)
:=
\coprod_{w\in\Trop(Y)(\ZZ)}
\pi_0\bigl(\init_w(Y)\bigr).
\]
Thus, an element of $\mathcal T(Y)$ is a pair $(w,[c])$, where
$w\in\Trop(Y)(\ZZ)$ and
\[
[c]\in\pi_0\bigl(\init_w(Y)\bigr)
\]
is a connected component of the corresponding initial variety. There is a natural projection
\[
\mathcal T(Y)\longrightarrow\Trop(Y)(\ZZ),
\qquad
(w,[c])\longmapsto w.
\]
\end{defn}

\begin{defn}
A very affine variety $Y\subset T$ is called \emph{schön} if, for every
$w\in\Trop(Y)$, the initial variety $\init_w(Y)$ is smooth.
\end{defn}

By a result of Helm and Katz \cite[Proposition 3.9]{HK} this is equivalent to the usual notion of a schön variety. If $Y$ is a smooth schön variety, then the Gröbner degeneration $\cY_{w} \to \AA^1$ is a smooth morphism.

\begin{defn}  For an affine $\CC[[t]]$-scheme $X$, its infinite Greenberg functor is defined by
\[
\Gr_\infty(X)(R)
:=
\Hom_{\Spec\CC[[t]]}
\bigl(\Spec R[[t]],X\bigr).
\]
It is represented by an affine $\CC$-scheme, in general not of finite type. 
\end{defn} 

Note that
\[
\Gr_\infty(X)=\varprojlim_n\Gr_n(X),
\]
where
\[
\Gr_n(X)(R)
=
\Hom_{\Spec\CC[[t]]}
\bigl(\Spec R[t]/(t^{n+1}),X\bigr).
\]
When $X$ is of finite type over $\CC[[t]]$, each $\Gr_n(X)$ is an affine scheme of finite type over $\CC$.
We will be interested in the Greenberg functor applied to the Gröbner degeneration, so set 
 \[
\widehat{\cY}_w
:=
\cY_w\times_{\AA^1}\Spec\CC[[t]].
\]

The main result of this subsection is the following.

\begin{theorem}\label{thm:tropical-fiber-main}
Let $Y$ be a very affine variety.
\begin{enumerate}
\item For every $w\in\Trop(Y)(\ZZ)$, we have
\[
L_wY\cong\Gr_\infty(\widehat{\cY}_w).
\]
In particular, $L_wY$ is an affine scheme, in general not of finite type over $\CC$. Moreover, $L_wY(\CC)=\val^{-1}(w)$.

\item For every $w\in\Trop(Y)(\ZZ)$, there is a natural morphism, called the \emph{initial-term morphism},
\[
\init_w:L_wY\longrightarrow\init_w(Y).
\]
If $Y$ and $\init_w(Y)$ are smooth and $d=\dim Y$, then $\init_w$ is a locally trivial fibration with fiber $(\AA^\infty)^d$, where $\AA^\infty=\Spec\CC[x_1,x_2,\ldots]$.

\item For each $w\in X_*(T)$, let $\Gr_T^w$ denote the corresponding connected component of the affine Grassmannian $\Gr_T$. Then there is a decomposition of ind-schemes
\[
LY=
\coprod_{w\in\Trop(Y)(\ZZ)}
\left(LY\times_{\Gr_T}\Gr_T^w\right).
\]
Moreover, for every $w\in\Trop(Y)(\ZZ)$, the natural morphism
$L_wY\to LY\times_{\Gr_T}\Gr_T^w$ induces a homeomorphism on underlying topological spaces. Consequently,
\[
|LY|=\coprod_{w\in\Trop(Y)(\ZZ)}|L_wY|.
\]

\item If $Y$ is schön, the initial-term morphisms induce a natural bijection
\[
\pi_0(LY)\xrightarrow{\sim}\mathcal T(Y).
\]
Equivalently,
\[
\pi_0(LY)
\cong
\coprod_{w\in\Trop(Y)(\ZZ)}
\pi_0\bigl(\init_w(Y)\bigr).
\]
\end{enumerate}
\end{theorem}

\begin{remark} \label{r:KU}
The ingredients of this theorem are largely present in the literature, although not in the form stated here. A closely related identification of fibers of tropicalization with Greenberg schemes of translated models appears, on field valued points, in \cite[Proposition 5.16]{KU}; the representability and smooth structure results used in parts \textup{(1)} and \textup{(2)} belong to the general theory of Greenberg schemes developed in \cite{GreenbergI,GreenbergII} and \cite{CNS}. The infinitesimal structure of the affine Grassmannian of a torus underlying part \textup{(3)} is described in \cite[\S 3.a]{PR}. The contribution of the theorem is to assemble these results in a functorial formulation for the intrinsic torus of an arbitrary very affine variety, to distinguish the valuation stratum from the corresponding ind-scheme component, and to deduce the description in part \textup{(4)} of the connected components of the loop space. To the best of our knowledge, this final description has not previously been recorded.
\end{remark} 

Before proving the theorem, we deduce a consequence for the algebraic loop map for schön varieties.

\begin{cor}\label{cor:constant-components} Let $Y$ be a schön very affine variety. Then the algebraic loop map is constant on connected components.
\end{cor}
\begin{proof}
Every connected component of $LY$ is contained in a unique
\[
LY\times_{\Gr_T}\Gr_T^w,
\qquad w\in\Trop(Y)(\ZZ).
\]
By Theorem \ref{thm:tropical-fiber-main}\textup{(3)}, the natural morphism
\[
L_wY\longrightarrow LY\times_{\Gr_T}\Gr_T^w
\]
induces a homeomorphism on underlying topological spaces. It also induces a bijection on complex points, since every morphism $\Spec\CC\to\Gr_T^w$ factors through the reduced point $\{t^w\}$. Thus, if $f,g\in LY(\CC)$ lie in the same connected component, then for a unique $w\in\Trop(Y)(\ZZ)$ they determine points $\widetilde f,\widetilde g\in L_wY(\CC)$ lying in the same connected component of $L_wY$.

Since $Y$ is schön, Theorem
\ref{thm:tropical-fiber-main}\textup{(2)} shows that
\[
\init_w:L_wY\longrightarrow\init_w(Y)
\]
is a Zariski-locally trivial fibration with fiber
$(\AA^\infty)^{\dim Y}$. Let
\[
c_f:=\init_w(\widetilde f),
\qquad
c_g:=\init_w(\widetilde g).
\]
Then $c_f$ and $c_g$ lie in the same connected component of $\init_w(Y)$. Since $\init_w(Y)$ is a variety, they can be joined by a chain of connected algebraic curves. After subdividing this chain so that each piece lies in an open set over which $\init_w$ is trivial, we may lift it to a chain of connected schemes in $L_wY$ joining $\widetilde f$ and $\widetilde g$. Composing with
\[
L_wY\longrightarrow LY
\]
gives a chain of connected algebraic families in $LY$ joining $f$ and $g$. The result now follows from Corollary \ref{cor:constancy-components}.
\end{proof}

\subsection{Proof of Theorem \ref{thm:tropical-fiber-main}}
To prove the theorem, we begin by reformulating the valuation condition in terms of arcs in the intrinsic torus $T$.

\begin{lemma}\label{lem:tropalg}
Let $f\in T(F)$ and $w\in X_*(T)$. Then $\val_f=w$ if and only if $t^{-w}f\in T(\CC[[t]])$.
\end{lemma}

\begin{proof}
Recall that the valuation of $f$ is the cocharacter $\val_f\in X_*(T)=\Hom(X^*(T),\ZZ)$ characterised by
$\langle\chi,\val_f\rangle=\val_t\bigl(\chi(f)\bigr)$ for every character $\chi\in X^*(T)$. Here, we regard $w$ as the $F$-point $w(t)\in T(F)$ and write $t^{-w}:=w(t)^{-1}$.

For every $\chi\in X^*(T)$, we have
$\chi(t^{-w}f)=t^{-\langle\chi,w\rangle}\chi(f)$. Hence
\[
\val_t\bigl(\chi(t^{-w}f)\bigr)
=
\val_t\bigl(\chi(f)\bigr)-\langle\chi,w\rangle
=
\langle\chi,\val_f-w\rangle.
\]
Now $t^{-w}f\in T(\CC[[t]])$ if and only if
$\chi(t^{-w}f)\in\CC[[t]]^\times$ for every $\chi\in X^*(T)$. This is equivalent to $\langle\chi,\val_f-w\rangle=0$ for every character $\chi\in X^*(T)$. Since the canonical pairing $X^*(T)\times X_*(T)\to\ZZ$ is perfect, this holds if and only if $\val_f=w$.
\end{proof}

The following lemma relates this condition to the Gröbner degeneration.

\begin{lemma}\label{lem:grobner-extension}
Let $Y\subset T$ be a very affine variety, let $w\in\Trop(Y)(\ZZ)$, and let $\cY_w\to\AA^1$ be the Gröbner degeneration of $Y$ with respect to $w$. For every $\CC$-algebra $R$ and every $f\in Y(R(\!(t)\!))$, the following conditions are equivalent.

\begin{enumerate}
\item The renormalised loop $t^{-w}f\in T(R(\!(t)\!))$ extends to an $R[[t]]$-point of the torus.

\item The morphism $\Spec R(\!(t)\!)\to Y\times\Gm\cong\cY_w|_{\Gm}$ induced by $t^{-w}f$ extends to a morphism $\Spec R[[t]]\to\cY_w$ over $\AA^1$.
\end{enumerate}
\end{lemma}

\begin{proof}
By construction, the Gröbner degeneration is a closed subscheme $\cY_w\subset T\times\AA^1$ whose restriction over $\Gm$ is identified with $Y\times\Gm$ by the morphism
\[
Y\times\Gm\longrightarrow\cY_w|_{\Gm},
\qquad
(y,t)\longmapsto(t^{-w}y,t).
\]
The implication \textup{(2)}$\Rightarrow$\textup{(1)} is therefore immediate: compose the extension $\Spec R[[t]]\to\cY_w$ with the closed immersion $\cY_w\hookrightarrow T\times\AA^1$ and project to $T$.

Conversely, suppose that $t^{-w}f\in T(R[[t]])$. Together with the standard morphism $\Spec R[[t]]\to\AA^1$ induced by $\CC[t]\to R[[t]]$, this gives a morphism $\widetilde f:\Spec R[[t]]\to T\times\AA^1$. Its restriction to $\Spec R(\!(t)\!)$ factors through $\cY_w|_{\Gm}\cong Y\times\Gm$.

It remains to show that $\widetilde f$ factors through the closed subscheme $\cY_w$. Let $I_w$ be the ideal defining $\cY_w$. For every $g\in I_w$, the pullback $\widetilde f^{\,*}(g)\in R[[t]]$ vanishes after inverting $t$, since the restriction of $\widetilde f$ to $\Spec R(\!(t)\!)$ factors through $\cY_w|_{\Gm}$. The natural map $R[[t]]\to R(\!(t)\!)=R[[t]][t^{-1}]$ is injective, because multiplication by $t$ is injective on $R[[t]]$. Hence $\widetilde f^{\,*}(g)=0$ for every $g\in I_w$. Therefore, $\widetilde f$ factors through $\cY_w$, giving the required extension $\Spec R[[t]]\to\cY_w$.
\end{proof}

We are now ready to prove the theorem.

\begin{proof}[Proof of Theorem \ref{thm:tropical-fiber-main}]
For (1), note that for every $\CC$-algebra $R$,
\[
L_wY(R)
=
\left\{
f\in Y(R(\!(t)\!))
\;\middle|\;
t^{-w}f\in\cY_w(R[[t]])
\right\}.
\]
Since $\cY_w|_{\Spec\CC(\!(t)\!)}\cong
Y\times\Spec\CC(\!(t)\!)$, giving such an $f$ is equivalent to giving a morphism $\Spec R[[t]]\to\widehat{\cY}_w$ over $\Spec\CC[[t]]$. Hence
\[
L_wY(R)
=
\Hom_{\Spec\CC[[t]]}
\bigl(\Spec R[[t]],\widehat{\cY}_w\bigr)
=
\Gr_\infty(\widehat{\cY}_w)(R).
\]
Therefore, $L_wY\cong\Gr_\infty(\widehat{\cY}_w)$ by Yoneda's lemma. In particular, $L_wY$ is represented by an affine scheme, in general not of finite type over $\CC$. Combining Lemmas \ref{lem:tropalg} and \ref{lem:grobner-extension}, we obtain $L_wY(\CC)=\val^{-1}(w)$. This concludes the proof of part (1).

For (2), under the identification $L_wY\cong\Gr_\infty(\widehat{\cY}_w)$, the natural truncation morphism $\Gr_\infty(\widehat{\cY}_w)\to\Gr_0(\widehat{\cY}_w)$ is given on $R$-points by reduction modulo $t$. Since
\[
\Gr_0(\widehat{\cY}_w)
\cong
\widehat{\cY}_w
\times_{\Spec\CC[[t]]}\Spec\CC
\cong
\init_w(Y),
\]
this defines the initial-term morphism $\init_w:L_wY\to\init_w(Y)$.

Assume now that $Y$ and $\init_w(Y)$ are smooth. Since the Gröbner degeneration $\cY_w\to\AA^1$ is flat and of finite presentation, its fibers over $\Gm$ are isomorphic to $Y$, while its fiber over $0$ is $\init_w(Y)$. Thus, all its fibers are smooth, and consequently $\cY_w\to\AA^1$ is smooth. It follows that $\widehat{\cY}_w\to\Spec\CC[[t]]$ is smooth. The result now follows from \cite[Chapter 5, \S 1.2, Proposition 1.2.2]{CNS}: the morphism $\init_w:L_wY\to\init_w(Y)$ is a locally trivial fibration with fiber $(\AA^\infty)^d$, where $d=\dim Y$. This proves (2).

For (3), the closed immersion $Y\hookrightarrow T$ induces an ind-closed immersion $LY\hookrightarrow LT$ and hence a morphism $LY\to\Gr_T$, where $\Gr_T=LT/L^+T$ is the affine Grassmannian of $T$. For each $w\in X_*(T)$, let $\Gr_T^w$ denote the corresponding connected component of $\Gr_T$. Since
\[
\Gr_T=\coprod_{w\in X_*(T)}\Gr_T^w,
\]
base change along $LY\to\Gr_T$ gives a decomposition of ind-schemes
\[
LY=\coprod_{w\in X_*(T)}
\left(LY\times_{\Gr_T}\Gr_T^w\right).
\]
The factor indexed by $w$ is empty unless $w\in\Trop(Y)(\ZZ)$, since any geometric point of this factor determines a loop in $Y$ with valuation $w$. Hence
\[
LY=\coprod_{w\in\Trop(Y)(\ZZ)}
\left(LY\times_{\Gr_T}\Gr_T^w\right).
\]

We now claim that $L_wY=LY\times_{\Gr_T}\{t^w\}$. Indeed, let $R$ be a $\CC$-algebra and let $f\in Y(R(\!(t)\!))$. Since $\Gr_T=LT/L^+T$ is the fpqc sheaf quotient, the equality $[f]=t^w$ in $\Gr_T(R)$ is equivalent to the existence of a faithfully flat extension $R\to R'$ such that $t^{-w}f\in T(R'[[t]])$. This condition descends to $R$, so it is equivalent to $t^{-w}f\in T(R[[t]])$. By Lemma \ref{lem:grobner-extension}, this is equivalent to $t^{-w}f\in\cY_w(R[[t]])$, which is precisely the defining condition for $f\in L_wY(R)$.

Finally, $(\Gr_T^w)_{\red}=\{t^w\}$, so the reduction morphism $\{t^w\}\to\Gr_T^w$ is a universal homeomorphism. Its base change along $LY\to\Gr_T$ is therefore also a universal homeomorphism. Since this base change is the natural morphism $L_wY\to LY\times_{\Gr_T}\Gr_T^w$, the latter induces a homeomorphism on underlying topological spaces. This proves (3).

For (4), suppose $Y$ is schön. By part \textup{(3)}, we have
\[
\pi_0(LY)
\cong
\coprod_{w\in\Trop(Y)(\ZZ)}
\pi_0(L_wY).
\]
Moreover, $Y$ and every initial variety $\init_w(Y)$ are smooth. By part \textup{(2)}, the initial-term morphism $\init_w:L_wY\to\init_w(Y)$ is a locally trivial fibration with fiber $(\AA^\infty)^{\dim Y}$. In particular, it induces a bijection
\[
\pi_0(L_wY)
\xrightarrow{\sim}
\pi_0\bigl(\init_w(Y)\bigr).
\]
Taking the disjoint union over all $w\in\Trop(Y)(\ZZ)$ gives
\[
\pi_0(LY)
\cong
\coprod_{w\in\Trop(Y)(\ZZ)}
\pi_0\bigl(\init_w(Y)\bigr)
=
\mathcal T(Y).
\]
This concludes the proof of (4).
\end{proof}

\section{Algebraic loops in smooth very affine varieties} \label{s:alg-loops-very-affine}
In this section, we discuss the relationship between tropicalization and the algebraic loop map constructed in Theorem \ref{thm:Lalg}. 

\subsection{Algebraic loops and the valuation map} Let $Y$ be a very affine variety and choose an embedding $Y\hookrightarrow T$ into the intrinsic torus. We then obtain a homomorphism 
\[
\pi_1(Y)\ra \pi_1(T)=X_*(T).
\]
 As the target is commutative, this homomorphism is constant on conjugacy classes and induces a map 
 \[
 \pi_1(Y)/\!\!\sim\,\, \ra X_*(T).
 \]
  This map is independent of the chosen embedding. 
\begin{prop}\label{prop:tropwinding}
The following diagram is commutative:
\[
\xymatrix{
Y(F) \ar[r]^{\cL_{\alg}} \ar[dr]^{\val \quad} & \pi_{1}(Y)/\sim \ar[d] \\
& \qquad X_*(T).
}
\]

\end{prop}
\begin{proof} 
By naturality of the algebraic loop map in Corollary \ref{cor:naturality}, the closed immersion $i:Y\hookrightarrow T$ induces a commutative diagram
\[
\xymatrix{
Y(F) \ar[r]^{\cL_{\alg}} \ar[d]_{i} &
[S^1,Y^{\an}] \ar[d]^{i_*} \\
T(F) \ar[r]^{\cL_{\alg,T}} &
[S^1,T^{\an}].
}
\]
As seen in Example \ref{expl:algloops} (3), under the identification
\[
[S^1,T^{\an}] \cong X_*(T),
\]
the map $\cL_{\alg,T}$ is the valuation map. Consequently, the composite
\[
Y(F)\longrightarrow T(F)\longrightarrow X_*(T)
\]
is precisely $\val:Y(F)\to X_*(T)$.
\end{proof}

\begin{cor}\label{cor:loop-determines-valuation}
Let $Y$ be a smooth connected very affine variety. If two formal loops
$f,g\in Y(F)$ determine the same algebraic loop, then they have the same
tropicalization:
\[
\cL_{\alg}(f)=\cL_{\alg}(g)
\qquad\Longrightarrow\qquad
\val_f=\val_g.
\]
In particular, algebraic loops arising from distinct integral tropical
points are distinct.
\end{cor}

\begin{proof}
By the above Proposition, the valuation of a formal loop is
the image of its algebraic loop under the natural map
$\pi_1(Y^{\an})/\!\sim\;\longrightarrow X_*(T)$. 
The result follows immediately.
\end{proof}

\begin{exam} In the situation of Example \ref{ex:hyparr}, we are given a basis of $\cO(M)^{\times} / \CC^{\times}$, up to individual scaling of the defining linear forms $\ell_{i}$. This scaling has no influence on valuations, and given a formal loop $f : \cD^{\times} \to M$ we get a canonical valuation vector 
\[
w_f=\bigl(\val_t(\ell_1(f)),\ldots,\val_t(\ell_n(f))\bigr).
\]
 We call this vector the \emph{hyperplane valuation function}, in analogy to the root valuation functions defined in \cite{GKM}.
\end{exam}

\subsection{Algebraic loops and Gröbner degenerations}
In this subsection, we explain how the Gröbner degeneration associated with the valuation of a formal loop encodes its leading-order behavior and its associated topological loop. Although a formal loop $f\in Y(F)$ need not extend to $t=0$ as a map to $Y$, it extends uniquely to an arc $\G_f$ in the Gröbner degeneration $\cY_{w_f}$. Its value at the origin defines a central point
\[
c_f\in\init_{w_f}(Y),
\]
which records the leading coefficients of $f$. We show that sufficiently accurate convergent approximations of $\G_f$ recover the algebraic loop $\g_f=\cL_{\alg}(f)$. We then describe the image of this loop in the intrinsic torus: after removing the radial growth determined by $w_f$, the loop deforms to the cocharacter loop $\theta\mapsto(e^{i\theta})^{w_f}c_f$.

\begin{prop}\label{p:degenloop}
Let $Y$ be a smooth very affine variety, let $f\in Y(F)$, and put
$w=\val(f)$. Let
\[
(f,t):\cD^\times\longrightarrow Y\times\Gm
\]
be the induced morphism. Then:
\begin{enumerate}
\item The morphism $(f,t)$ extends uniquely to an arc $\G_f:\cD\longrightarrow\cY_w$ over $\AA^1$.

\item For all sufficiently large $N$, any order-$N$ algebraic
approximation
\[
\G_N:\cD\longrightarrow\cY_w
\]
of $\G_f$ over $\AA^1$ restricts, on a sufficiently small punctured
analytic disk, to a map
\[
\G_N^{\an}:S_r^1\longrightarrow Y^{\an}\times\CC^\times,
\qquad
t\longmapsto (p_N(t),t),
\]
where the free homotopy class of $p_N$ is
\[
[p_N]=\cL_{\alg}(f).
\]
\end{enumerate}
\end{prop}

\begin{proof}
Part \textup{(1)} is Lemma \ref{lem:grobner-extension}.

For \textup{(2)}, choose $N$ sufficiently large that the approximation $\G_N$ computes the algebraic loop associated with the formal arc $\G_f$ and such that, on a sufficiently small punctured disk, its image is contained in the open subscheme
\[
Y\times\Gm\subset\cY_w.
\]
Since $\G_N$ is a morphism over $\AA^1$, its restriction to
$S_r^1$ is necessarily of the form
\[
t\longmapsto (p_N(t),t).
\]
On the generic fiber, $\G_f$ is the formal loop
\[
(f,t):\cD^\times\longrightarrow Y\times\Gm.
\]
Hence the loop represented by $\G_N^{\an}$ is
\[
\cL_{\alg,Y\times\Gm}(f,t).
\]
By naturality of the algebraic loop map with respect to the projection
$\pr_Y:Y\times\Gm\to Y$,
\[
(\pr_Y)_*
\cL_{\alg,Y\times\Gm}(f,t)
=
\cL_{\alg,Y}(f).
\]
The left-hand side is represented by $p_N$, and therefore $[p_N]=\cL_{\alg}(f)$.
\end{proof}

We now give an explicit description of the algebraic loop associated
with a formal loop in terms of the corresponding initial variety.
Let $f\in Y(F)$ have valuation $w=w_f$ and let
\[
\G_f:\cD\longrightarrow \cY_w
\]
be the arc in the Gröbner degeneration constructed above, with central
point
\[
c=c_f:=\G_f(0)\in\init_w(Y).
\]
The idea is to compute the limit, after renormalization by $w$, of a
convergent approximation of $\G_f$. 
\begin{lemma}\label{lem:central-loop}
After replacing $f$ by a sufficiently accurate convergent approximation,
let
\[
f:\Delta_{r_0}^{\times}\longrightarrow Y^{\an}
\]
denote the corresponding holomorphic map. For every $0<r_1<r_0$, the map
\[
H:[0,r_1]\times S^1\longrightarrow T^{\an},
\]
defined by
\[
H(r,e^{i\theta})=
\begin{cases}
r^{-w}f(re^{i\theta}), & r>0,\\[1ex]
(e^{i\theta})^w c, & r=0,
\end{cases}
\]
is continuous. 
\end{lemma}
\begin{proof}
Write
\[
t^{-w}f=u,
\]
where $u$ extends holomorphically across $0$ with $u(0)=c$. Then
\[
r^{-w}f(re^{i\theta})
=(e^{i\theta})^wu(re^{i\theta}).
\]
Since $u$ is continuous, the map
\[
(\theta,r)\longmapsto (e^{i\theta})^wu(re^{i\theta})
\]
extends continuously to $r=0$, where it takes the value $(e^{i\theta})^wc$.
\end{proof}

A local triviality result for the Gröbner degeneration then allows us to transport the loop $\g_f$ to the special fiber $\init_{w}(Y)$ where it is represented by the cocharacter loop $\theta \mapsto (e^{i\theta})^wc$.

\begin{theorem}\label{thm:initialloop}
Assume that $Y$ and $\init_w(Y)$ are smooth, and let $K\subset \init_w(Y)^{\an}$ be the compact image of the cocharacter loop
\[
e^{i\theta}\longmapsto (e^{i\theta})^wc_f.
\]

Then there exist an open neighborhood $U_0\subset \init_w(Y)^{\an}$ of $K$, a disk $\Delta\subset\CC$ centered at $0$, and an open subset $U\subset \mathcal Y_w^{\an}$ containing $U_0$ such that the restriction 
\[
p|_U:U\longrightarrow \Delta
\]
of the Gröbner degeneration $p : \mathcal Y_{w}^{\an} \to \CC$ admits a trivialization
\[
U\simeq U_0\times\Delta
\]
over $\Delta$. For every sufficiently small $r>0$, under the induced identification
\[
\pi_1(U_r)\simeq \pi_1(U_0),
\qquad U_r:=U\cap p^{-1}(r),
\]
the algebraic loop $\gamma_f=\cL_{\alg}(f)$, viewed in the fiber $p^{-1}(r)$ via the canonical identification
\[
Y^{\an}\simeq p^{-1}(r),
\]
corresponds to the class of the loop
\[
S^1\longrightarrow \init_w(Y)^{\an},
\qquad
e^{i\theta}\longmapsto (e^{i\theta})^wc_f.
\]
\end{theorem}

\begin{proof}
Under our assumptions, the Gröbner degeneration
\[
p:\mathcal Y_w\longrightarrow\AA^1
\]
is smooth, hence its analytification is a submersion. By Proposition~\ref{p:degenloop}, we may compute $\gamma_f$ using a
sufficiently accurate convergent approximation of $\G_f$. By
Lemma~\ref{lem:central-loop}, after shrinking $\delta>0$ if necessary,
the corresponding rescaled loops define a continuous map
\[
P:S^1\times[0,\delta)\longrightarrow\mathcal Y_w^{\an}
\]
over $[0,\delta)$, given for $r>0$ by
\[
P(e^{i\theta},r) = \bigl(r^{-w}f(re^{i\theta}),r\bigr),
\]
and at $r=0$ by
\[
P(e^{i\theta},0)
=
\bigl((e^{i\theta})^wc_f,0\bigr).
\]
For $r>0$, under the canonical identification
\[
Y^{\an}\simeq p^{-1}(r),\qquad
y\longmapsto (r^{-w}y,r),
\]
the loop $P(-,r)$ represents $\gamma_f$. Now apply Lemma~\ref{lem:loctrivcompact} to the compact set
\[
K=P(S^1,0)\subset p^{-1}(0)=\init_w(Y)^{\an}.
\]
After shrinking $\delta$, we obtain an open neighborhood
$U_0\subset\init_w(Y)^{\an}$ of $K$, an open subset
$U\subset\mathcal Y_w^{\an}$, and a trivialization $U\simeq U_0\times\Delta$ over a disk $\Delta$ about $0$. By Corollary~\ref{cor:centloop}, the loop $P(-,r)\subset U_r$ is homotopic to the parallel transport of $P(-,0)$ for every sufficiently small $r>0$. Hence, under the induced identification $\pi_1(U_r) \simeq\pi_1(U_0)$, the class $\gamma_f$ corresponds to the class of $e^{i\theta}\longmapsto(e^{i\theta})^wc_f$.
\end{proof}
Recall that the enhanced integral tropicalization is
\[
\cT(Y)
=
\coprod_{w\in\Trop(Y)(\ZZ)}
\pi_0\bigl(\init_w(Y)\bigr).
\]
For $f\in Y(F)$, define
\[
\nu(f):=(w_f,[c_f])\in\cT(Y).
\]

\begin{cor}\label{cor:troplift}
Let $Y$ be a schön very affine variety. Then there is a natural
surjective map
\[
\cL_{\trop}:
\cT(Y)\longrightarrow [S^1,Y^{\an}]^{\alg}
\]
such that
\[
\xymatrix{
Y(F) \ar[r]^{\nu}\ar[dr]_{\cL_{\alg}}
& \cT(Y) \ar[d]^{\cL_{\trop}} \\
& [S^1,Y^{\an}]^{\alg}.
}
\]
For $(w,[c])\in\cT(Y)$, the class
$\cL_{\trop}(w,[c])$ is represented, after transport to the special
fiber $\init_w(Y)$, by
\[
e^{i\theta}\longmapsto(e^{i\theta})^wc,
\]
for any $c$ in the component $[c]$.
\end{cor}

\begin{proof}
By Theorem \ref{thm:tropical-fiber-main}\textup{(4)}, there is a
natural bijection
\[
\pi_0(LY)\xrightarrow{\sim}\cT(Y).
\]
By Corollary \ref{cor:constant-components}, the algebraic loop map is
constant on connected components of $LY$, and therefore factors as
\[
Y(F)=LY(\CC)
\longrightarrow
\pi_0(LY)
\longrightarrow
[S^1,Y^{\an}]^{\alg}.
\]
Transporting the second map across the above identification defines
$\cL_{\trop}$ and gives the commutative diagram. It is surjective by
definition of $[S^1,Y^{\an}]^{\alg}$. The description in the special fiber follows from Theorem \ref{thm:initialloop}.
\end{proof}

\begin{theorem}\label{thm:trop-classification}
Assume that $Y$ is schön and that $\init_w(Y)$ is connected for every
$w\in\Trop(Y)$. Then
\[
\cL_{\trop}:
\pi_0(LY)
\cong
\Trop(Y)(\ZZ)
\xrightarrow{\sim}
[S^1,Y^{\an}]^{\alg}
\]
is a bijection. The class corresponding to $w$ is represented, after transport to the
special fiber $\init_w(Y)$, by
\[
e^{i\theta}\longmapsto(e^{i\theta})^wc,
\]
for any $c\in\init_w(Y)$.
\end{theorem}

\begin{proof}
Since every initial variety is connected,
\[
\cT(Y)=\Trop(Y)(\ZZ).
\]
By Corollary \ref{cor:troplift}, the map $\cL_{\trop}$ is surjective,
so it remains to prove injectivity. But this follows by Proposition
\ref{prop:tropwinding}, because the image of an algebraic loop under
\[
[S^1,Y^{\an}]
\longrightarrow
[S^1,T^{\an}]
\cong X_*(T)
\]
is its valuation.
\end{proof}

\begin{remark}
Schön very affine varieties with connected initial varieties are
sometimes called \emph{wunderschön}. Thus, for a wunderschön variety,
algebraic loops are completely classified by integral tropical points.
Equivalently, two formal loops determine the same algebraic loop if
and only if they lie in the same connected component of $LY$.
Complements of linear hyperplane arrangements are wunderschön, as we
will recall in Section~\ref{s:hyparr}.
\end{remark}

\subsection{Tropical compactifications and the ray formula}

In this subsection, we briefly explain how to compute the algebraic loop $\g_f$ attached to $f$ from a tropical compactification. The polyhedral structure of $\Trop(Y)$ is used to identify the relevant boundary meridians, whose powers can be read off from the valuation vector. We use the standard theory of tropical compactifications of subvarieties of tori; see \cite{TevelevCompactifications}. The rays of the tropicalization correspond to
boundary divisors, while the coefficients of a tropical point in a cone record the contact orders of a formal loop with these divisors. Combining this with the meridian description of algebraic loops gives an explicit expression for $\g_f=\cL_{\alg}(f)$ as a product of commuting boundary meridians, which we call the ray formula.

Let $Y$ be a schön very affine variety embedded in its intrinsic torus
$T$, and set
\[
N_{\RR}:=X_*(T)\otimes_{\ZZ}\RR.
\]
Choose a smooth rational polyhedral fan $\Sigma$ in $N_{\RR}$ with
\[
\operatorname{supp}(\Sigma)=\Trop(Y)
\]
such that the closure $\ov Y_\Sigma$ of $Y$ in the associated toric variety $X_\Sigma$ is a smooth proper compactification with simple normal crossings boundary.  We write $\Sigma(1)$ for the set of rays of $\Sigma$. For $\rho\in\Sigma(1)$, let
\[
v_\rho\in X_*(T)
\]
be its primitive generator. For a cone $\sigma\in\Sigma$, we denote by
$\sigma(1)\subset\Sigma(1)$ the set of its rays and by
\[
\cO_\sigma\subset X_\Sigma
\]
the corresponding torus orbit.

For $f\in Y(F)$, let $w_f=\val(f)\in\Trop(Y)(\ZZ)$ and denote by $\sigma_f\in\Sigma$ the unique cone whose relative interior contains $w_f$. Since $\Sigma$ is smooth, there are unique integers $a_\rho(f)>0$, for $\rho\in\sigma_f(1)$, such that
\[
w_f=\sum_{\rho\in\sigma_f(1)}a_\rho(f)v_\rho.
\]

We first relate this decomposition to the boundary of the tropical compactification. The boundary strata of $\ov Y_\Sigma$ are the intersections
\[
\ov Y_\Sigma\cap\cO_\sigma,
\qquad \sigma\in\Sigma,
\]
and tropicality implies
\[
\codim_{\ov Y_\Sigma}
\bigl(\ov Y_\Sigma\cap\cO_\sigma\bigr)
=
\dim\sigma.
\]
In particular, the codimension-one boundary strata correspond to the
rays of $\Sigma$. These strata need not be connected.

\begin{lemma}\label{lem:boundary-components}
The irreducible components of the boundary
\[
D=\ov Y_\Sigma\setminus Y
\]
are naturally parametrized by pairs
\[
(\rho,[c]),
\qquad
\rho\in\Sigma(1),\quad
[c]\in\pi_0(\ov Y_\Sigma\cap\cO_\rho).
\]
If all initial varieties of $Y$ are connected, then every
$\ov Y_\Sigma\cap\cO_\rho$ is connected, and hence the irreducible
components of $D$ are naturally indexed by the rays of $\Sigma$.
\end{lemma}

\begin{proof}
The first assertion follows from the description of the boundary
strata above. For $w$ in the relative interior of a cone $\sigma$,
\cite[Lemma 3.6]{HK} identifies $\init_w(Y)$ with a torus bundle over
\[
\ov Y_\Sigma\cap\cO_\sigma.
\]
Since the fibers are connected, this induces a bijection
\[
\pi_0(\init_w(Y))
\cong
\pi_0(\ov Y_\Sigma\cap\cO_\sigma).
\]
The second assertion follows.
\end{proof}

Let now $f\in Y(F)$. By properness, $f$ extends uniquely to
\[
\ov f:\Spec\cO\longrightarrow\ov Y_\Sigma,
\]
and its center lies in the stratum
\[
\ov Y_\Sigma\cap\cO_{\sigma_f}.
\]
For each $\rho\in\sigma_f(1)$, let $D_{\rho,f}$ denote the unique
irreducible component of the boundary lying over $\rho$ and containing
the center of $\ov f$, and let $I_{\rho,f}$ be its ideal sheaf near
that point.

\begin{lemma}\label{lem:contact-orders}
With the notation above,
\[
a_\rho(f)
=
\nu_t\bigl(\ov f^*I_{\rho,f}\bigr),
\qquad
\rho\in\sigma_f(1).
\]
Thus the coefficients of $w_f$ with respect to the primitive ray
generators of $\sigma_f$ are precisely the contact orders of $\ov f$
with the corresponding boundary components.
\end{lemma}

\begin{proof}
For $\chi\in X^*(T)$, the toric divisor formula gives
\[
\operatorname{div}_{X_\Sigma}(\chi)
=
\sum_{\rho\in\Sigma(1)}
\langle\chi,v_\rho\rangle D_\rho^{\mathrm{tor}},
\]
where $D_\rho^{\mathrm{tor}}$ is the toric boundary divisor
corresponding to $\rho$; see \cite[Proposition 4.1.2]{CLS}.
Restricting to $\ov Y_\Sigma$ and pulling back along $\ov f$ gives
\[
\nu_t(\chi(f))
=
\sum_{\rho\in\sigma_f(1)}
\nu_t\bigl(\ov f^*I_{\rho,f}\bigr)
\langle\chi,v_\rho\rangle.
\]
On the other hand,
\[
\nu_t(\chi(f))
=
\langle\chi,w_f\rangle
=
\sum_{\rho\in\sigma_f(1)}
a_\rho(f)\langle\chi,v_\rho\rangle.
\]
Since this holds for every $\chi\in X^*(T)$ and the primitive
generators $v_\rho$, $\rho\in\sigma_f(1)$, are linearly independent,
we obtain
\[
a_\rho(f)
=
\nu_t\bigl(\ov f^*I_{\rho,f}\bigr)
\]
for every $\rho\in\sigma_f(1)$.
\end{proof}

We can now express the algebraic loop in terms of boundary meridians.

\begin{prop}[Ray formula]\label{p:rayformula}
Let $f\in Y(F)$ and write
\[
w_f
=
\sum_{\rho\in\sigma_f(1)}
a_\rho(f)v_\rho.
\]
Let $x_f$ be the center of $\ov f$. For each
$\rho\in\sigma_f(1)$, let $\gamma_{\rho,x}$ be a positively oriented
meridian around the boundary component $D_{\rho,f}$, transported to a
common base point $x\in Y^{\an}$.

The meridians $\gamma_{\rho,x}$ may be chosen to commute, and
\[
\cL_{\alg}(f)
=
\left[
\prod_{\rho\in\sigma_f(1)}
\gamma_{\rho,x}^{\,a_\rho(f)}
\right]
\]
as a conjugacy class in $\pi_1(Y^{\an},x)$.
\end{prop}

\begin{proof}
Since $D=\ov Y_\Sigma\setminus Y$ is a simple normal crossings
divisor, there is a sufficiently small analytic neighborhood $\ov U$
of $x_f$ in which the boundary components through $x_f$ are precisely
the $D_{\rho,f}$ for $\rho\in\sigma_f(1)$ and are given by coordinate
hyperplanes. Hence
\[
U:=\ov U\cap Y
\]
is analytically isomorphic to $(\Delta^{\times})^{r} \times \Delta^{n-r}$, so its fundamental group is generated by commuting meridians around the coordinate hyperplanes.

By Lemma \ref{lem:contact-orders}, the contact order of $\ov f$ with $D_{\rho,f}$ is $a_\rho(f)$. The meridian description of the algebraic loop therefore gives, after choosing a base point in $U$,
\[
\g_f
=
\prod_{\rho\in\sigma_f(1)}
\gamma_{\rho,x}^{\,a_\rho(f)}
\]
in $\pi_1(Y^{\an},x)$, up to simultaneous conjugation. Passing to the
conjugacy class proves the formula.
\end{proof}

\section{Linear hyperplane arrangements} \label{s:hyparr}

In this section, we let $V$ be a complex vector space, and $\cA$ a finite arrangement of linear hyperplanes in $V$. We refer to \cite{OrlikTerao} for background on hyperplane arrangements, their intersection lattices, circuits, and complements. We always assume the arrangement to be \emph{essential}, i.e. $\bigcap_{H\in \cA} H = 0$. This condition can always be achieved by passing to the quotient of $V$ by this intersection.

For a linear hyperplane arrangement $(V,\cA)$ we denote by 
\[M = M(V,\cA) := V \setminus \bigcup_{H \in \cA} H \]
the arrangement complement in $V$. We fix for every $H \in \cA$ a linear functional $\a_H\in V^*$ such that $H=\ker(\a_H)$, and we define $T=T(V,\cA) := \Gm^{\cA}$. This determines a closed immersion
\[ M \hookrightarrow \Gm^{\cA}, \, x \mapsto (\a_H(x))_{H\in \cA}.\]
We denote its image by 
\[ Y = Y(V,\cA), \]
to emphasize that $Y$ is embedded into $T$. In this situation, $T$ is identified with the intrinsic torus of the very affine variety $M\cong Y$.

\subsection{Hyperplane valuations and associated filtrations}

In analogy to root valuation functions and root valuation data \cite{GKM}, we make the following definition.

\begin{defn} Let $f\in Y(F)$. We call the valuation vector $w_f\in \Trop(Y)(\ZZ)$ a \emph{hyperplane valuation function}, and think of it as a map
\[ \cA \to \ZZ, \ H \mapsto d_f(H) := (w_f)_{H}. \]
Moreover, we associate to any $w\in \Trop(Y)(\ZZ)$ such a hyperplane valuation function $d=d_{w}$. We will use these points of view interchangeably.
\end{defn}
Our first goal is to describe initial varieties in terms of this function. To that end, we introduce a filtration on $\cA$ determined by a hyperplane valuation function.

\begin{defn} Let $w \in \Trop(Y)(\ZZ)$ with hyperplane valuation function $d$. The descending filtration on $\cA$ defined by
\[ \cA_{> j} := \{H\in\cA \mid d(H) > j \} \]
is called the \emph{valuation filtration} on $\cA$ with respect to $d$. Similarly, the ascending filtration of $V$ defined by 
\[ V_j = \bigcap_{H \in \cA_{>j}} H \] 
is called the \emph{valuation filtration} on $V$ with respect to $d$.
\end{defn}

By definition, $\bigcap_{j} \cA_{>j} = \emptyset$, and therefore the filtration $(V_j)$ is exhaustive. The next lemma follows directly from the construction (and since we assume $(V,\cA)$ is essential).
\begin{lemma}\label{lem:minmax}
Assume $j < \min\{d(H) \mid H\in \cA\}$. Then $\cA_{>j} = \cA$, and in turn $V_j = 0$. Moreover, when $j \ge \max\{d(H) \mid H\in \cA\}$, then $\cA_{>j} = \emptyset$, and in turn $V_j = V$.
\end{lemma}

\begin{remark} In the arrangement complement case, $f$ has an explicit description. It is an $F$-point of $V$ which avoids all hyperplanes, i.e. it is a Laurent series
\[ f = \sum_{j \in \ZZ, j \ge k} x_j t^{j} \]
such that $\alpha_{H}\circ f \in \AA^1(F)=F$ is non-zero for every $H\in \cA$. We sometimes call such a Laurent series a \emph{generically regular} Laurent series in $V$. We can describe the filtration $\cA_{>j}$ determined by $d_f$ in terms of the coefficients of $f$ as follows
\[ \cA_{> j} = \{H\in \cA \mid \alpha_{H}(x_{i}) = 0 \textup{ for all} i \le j \}. \]
For the ascending filtration on $V$ we have the description
\[ V_j = \bigcap_{\alpha_{H}(x_{i}) = 0 \textup{ for all} i \le j} \ker(\a_{H}).\]
\end{remark}

In the situation of a hyperplane arrangement complement, the valuation strata have a completely explicit description in terms of the valuation filtrations. We record it in the following, the proof is immediate.

\begin{lemma}
For $w\in \Trop(Y)(\ZZ)$ and for $j\in \ZZ$, define $\cA_{j} = \{V_j\cap H \mid H\in \cA, d(H)=j \}$. This is a hyperplane arrangement in the subspace $V_j$, and we define 
\[M_{j} := M(V_j,\cA_{j}) = V_{j} \setminus \bigcup_{H\in \cA_{j}} (H \cap V_{j}) \subset V.\]
Then the scheme $\prod_{j \in \ZZ} M_j$ is isomorphic to the valuation stratum $L_wY$, and a formal loop $f = \sum x_jt^{j}\in Y(F)$ has hyperplane valuation function $w$ if and only if $x_j \in M_j$ for all $j \in \ZZ$. 
\end{lemma}

Even though the valuation stratum $L_wY$ is infinite-dimensional, the next lemma shows that its homotopy type is that of a finite product of arrangement complements appearing between the minimal and maximal valuation levels.

\begin{lemma} Let $m_1 = \min\{d(H) \mid H\in \cA\}$ and $m_2=\max\{d(H) \mid H\in \cA\}$, and call \[\Adm(w) := \prod_{j = m_1}^{m_2} M_j \]
the \emph{admissible deformation space} of $w$. Then we have
\[L_wY \cong \Adm(w) \times \prod_{j > m_2} V,\]
i.e. $L_wY$ is a product of the admissible deformation space, which is a product of finitely many linear hyperplane arrangements, and an infinite-dimensional affine space. 
\end{lemma}
\begin{proof} This follows immediately from Lemma \ref{lem:minmax}: when $j < m_1$, then $V_j = 0$, and when $j > m_2$, then  $\cA_{j} = \emptyset$. The second statement follows from the definition of $M_j$.
\end{proof}

For $f \in Y(F)$, as long as we deform its coefficients inside $L_wY$, its hyperplane valuation function remains unchanged. In the case of irregular classes, the terminology admissible deformation space appears already in \cite{DRT}. Recall that we have an initial term map $L_wY \to \init_{w}Y$, which in the schön case is a locally trivial affine space bundle. We will describe this map explicitly in the next subsection. 

\subsection{Initial varieties and hyperplane valuation functions}

Let $w\in \Trop(Y)(\ZZ)$. We compute the initial variety $\init_{w}(Y)$ in terms of the valuation filtrations introduced above. Consider the ascending valuation filtration $(V_j)$ on $V$ determined by $w$, and define
\[ \gr_{w} V = \bigoplus_{j\in \ZZ} V_{j}/V_{j-1}.\]
We will show that $\init_{w}(Y)$ decomposes as a ``graded arrangement complement'' in $\gr_{w}(V)$, for arrangements in the associated graded pieces of the filtration $(V_j)$. 
\begin{defn}
Let $\alpha\in V^*$ be a linear form. Then we define the \emph{initial form} of $\a$ with respect to $w$ as follows. Let $j(\alpha)$ be the smallest index for which $\a$ does not vanish identically on $V_{j}$. When $\a = \a_H$, then $j(\a) = d(H)$. Then $\a$ vanishes on $V_{j(\alpha)-1}$, hence induces a linear form on $V_{j(\alpha)}/V_{j(\alpha)-1}$. The composition of this linear form with $\pr_j(\alpha) : \gr_{w}V \to V_{j(\alpha)}/V_{j(\alpha)-1}$ is
\[\init_{w}(\a)\in (\gr_w V)^*.\]
The hyperplanes $\ker(\init_{w}(\a_H)) \subset \gr_{w}V$ define a hyperplane arrangement in $\gr_w V$, denoted by $\gr_w \cA$. We call it the \emph{graded arrangement} for $w$. We write $\gr_{w}M := M(\gr_{w}V,\gr_{w}\cA)$ for the complement of the graded arrangement in $\gr_w V$.
\end{defn}

\begin{prop} \label{prop:initgr}
The map 
\begin{align*} 
\Phi_{w} : \gr_w M 	&\to \init_{w}(Y) \subset T  \\
		v		&\mapsto (\init_{w}(\a_{H}))_{H \in \cA},
\end{align*}
is a well-defined isomorphism of varieties. In other words, the initial variety of a hyperplane arrangement complement at $w$ is the complement of the graded arrangement associated with the valuation filtration.
\end{prop}
Before we prove this statement, we need to recall the notion of circuit to describe the defining equations of $Y$ and $\init_{w}(Y)$. 
\begin{defn}
A subset $C \subset \cA$ is called a \emph{circuit} if there is a linear relation $\sum_{H\in C} \l_H\a_H = 0$ in $V^*$ with $\l_H \in \CC^{\times}$ and $C$ is minimal with this property. For a circuit $C$, a \emph{circuit relation} is a nonzero relation
\[
\sum_{H\in C}\lambda_H\alpha_H=0
\]
with \(\lambda_H\in\mathbb C^\times\) for all \(H\in C\). Such a relation is unique up to multiplication by a scalar. The corresponding \emph{circuit polynomial} is
\[
\ell_C := \sum_{H\in C}\lambda_H x_H \in \CC[T],
\]
where \(x_H\) denotes the coordinate on the \(H\)-th factor of
\(T=\mathbb G_m^{\cA}\). The \emph{circuit ideal} is the ideal
\[
I_{\mathrm{circ}} := (\ell_C \mid C \subset \cA \text{ a circuit}) \subset \CC[T].
\]
\end{defn}
It is a standard result that the defining ideal of the hyperplane complement $Y$ in its intrinsic torus $T = \Gm^{\cA}$ is the circuit ideal $I_{\mathrm{circ}}$. 
\begin{lemma} \label{lem:initform}
Let $C$ be a circuit in $\cA$ with circuit polynomial $\ell_C = \sum_{H\in C}\lambda_H x_H$. Moreover, let $m_{C} = \min_{H \in C} \{d(H)\}$. Then the initial form of $\ell_C$ is described by 
\[\init_{w}(\ell_C) = \sum_{H \in C, d(H) = m_C} \l_H x_H,\] 
and we have  
\[\sum_{H \in C, d(H) = m_{C}} \l_H \init_w(\alpha_H) = 0\] 
for the initial forms of $\a_H$. In other words, the initial forms $\init_{w}(\a_H)$ satisfy the initial forms of circuit relations.
\end{lemma}
\begin{proof} The first statement is the standard description of initial forms. For the second description note that we have a circuit relation 
\[ \sum_{C} \l_H \a_H = 0.\]
Restricting this relation to $V_{m_{C}}$ kills all $\a_H$ with $d(H) > m_{C}$, and since the remaining $\a_H$ descend to $V_{m_{C}}/ V_{m_{C}-1}$, the claim follows from the definition of $\init_w(\a_H)$.
\end{proof}

\begin{proof}[Proof of Proposition \ref{prop:initgr}] We use the fact that the set of circuits determines a universal Gröbner base \cite[Proposition 1.6]{S}. Explicitly, this means that $\init_w(I_{\mathrm{circ}})$ is the ideal generated by all $\init_{w}(\ell_{C})$. We will use this at several points in the proof. Then Lemma \ref{lem:initform} shows that the map on coordinate rings
\[ \CC[T] \to \CC[\gr_w M], x_H \mapsto \init_{w}(\a_H) \]
factors through
\[ \varphi_w : \CC[T]/\init_{w}(I_{\mathrm{circ}}) \to \CC[\gr_w M] \]
and the map $\Phi_w$ is the corresponding map of varieties, which is therefore well-defined. We prove that $\varphi_w$ is an isomorphism. 

First we check surjectivity. For this we need to prove that the forms $\init_{w}(\a_H)$ generate $\CC[\gr_w M]$. Since $\gr_w M$ is precisely the complement of the hyperplane arrangement given by the $\init_{w}(\a_H)$, it suffices to prove that 
\[\mathrm{span}\{ \init_{w}(\a_H) \mid H \in \cA \} = (\gr_w V)^*.\] 
In other words, we need to prove that the graded arrangement is essential. For this we observe that
\[\bigcap_{d(H) = j} \ker(\init_{w}(\a_H)) = 0, \]
where $\init_{w}(\a_H)$ is understood as functional on $V_j / V_{j-1}$. This follows by definition of $V_j = \bigcap_{d(H) > j} H$. Indeed, if $v \in V_j$ with $\init_{w}(\a_H)(v) = \a_H(v) = 0$ for all $\a_H$ with $d(H) = j$, then by definition $v\in V_{j-1}$. This proves that $\init_{w}(\a_H)$ with $d(H) = j$ generate $(V_j/V_{j-1})^*$, hence $\gr_w(V)^* = \bigoplus_{j} (V_j/V_{j-1})^*$ is spanned by the forms $\init_{w}(\a_H)$. We conclude that $\varphi_{w}$ is surjective. 

We now check injectivity, in other words we prove that $\init_{w}(I_{\mathrm{circ}}) = \ker(\varphi_w)$. For this let $A:=\CC[x_H, H \in \cA]$ and consider the map
\[ f : A \to \CC[V],  x_H \mapsto \a_H. \]
Note that $w$ induces a grading on $A$ by letting $x_H$ have degree $d(H)$. Consider the induced descending filtration $F_{\ge m}A = \bigoplus_{j \ge m} A_j$. Then for $f \in A$ with minimum weight $m_0$, its class in $F_{\ge m_0}A / F_{\ge m_0+1} A$ is exactly its $w$-initial form. Moreover, for the associated graded $\gr A$ we have $\gr A \cong A$ as graded $\CC$-algebras.

On the other hand, we have
\[
\CC[V]=\Sym(V^\ast).
\]
The ascending filtration \(V_j\) on \(V\) induces a descending filtration on \(V^\ast\) by
\[
F_{\ge j}V^*:=\mathrm{Ann}(V_{j-1}),
\]
and hence a multiplicative descending filtration on \(\CC[V]\). For every \(j\), we have
\[
V_{j-1} = \bigcap_{d(H)\ge j}\ker(\alpha_H),
\]
and therefore
\[
F_{\ge j}V^* = \mathrm{Ann}(V_{j-1}) = \mathrm{span}\{\alpha_H\mid d(H)\ge j\}.
\]
It follows that the surjection
\[
f:A=\CC[x_H\mid H\in\mathcal A]\longrightarrow \CC[V], \qquad x_H\longmapsto \alpha_H,
\]
is strict with respect to the filtrations defined above. Hence passage to associated graded preserves its kernel, i.e. $\ker(\gr f)=\gr(\ker f)$. 

Under the natural identifications $\gr A\cong A$ and $\gr_w\CC[V]\cong \CC[\gr_w V]$,  the induced map 
\[\psi := \gr(f) :A\longrightarrow \CC[\gr_w V]\]
sends \(x_H\) to \(\init_w(\alpha_H)\), and $\gr(\ker f)=\init_w(\ker f)$.
Thus $\ker(\psi)=\init_w(\ker f)$. Now $\ker(f)$ is generated by the circuit polynomials in \(A\). Localizing \(\psi\) at the multiplicative system generated by the variables \(x_H\) gives \(\varphi_w\). Since localization is exact, we obtain $\ker(\varphi_w)=\init_w(I_{\mathrm{circ}})$.
\end{proof}
\begin{cor}
Under the identification $\Phi_{w}$, the initial term map $L_wY \to \init_{w}(Y)$ is identified with the natural map $\prod_j M_j \to \gr_w M$ induced from the projection $V_j \to V_j / V_{j-1}$. Moreover, it factors through
\[ \pi : \Adm(w) \to \gr_w M,\] 
which is an affine space bundle (that is globally trivializable). In particular, this map induces an isomorphism
\[ \pi_{1}(\Adm(w)^{\an}, x)  \to \pi_{1}(\init_{w}(Y)^{\an},\pi(x)). \]
\end{cor}
\begin{proof}
The first two claims follow from unravelling definitions. To see that $\pi$ is also an affine space bundle, for each $j$, one checks that the projection $V_j\to V_j/V_{j-1}$ restricts to a map
\[ M_j\to M(V_j/V_{j-1},\mathcal A_j), \]
whose fibers are affine spaces isomorphic to $V_{j-1}$. The bundle can be trivialized by picking splittings $V_j \cong V_{j-1} \oplus V_j/V_{j-1}$. The statement about fundamental groups is then immediate from Proposition \ref{prop:initgr}, because the fibers are contractible.
\end{proof}

\begin{remark} The set $\Trop(Y)$ has a natural structure of a polyhedral complex, induced from Gröbner degeneration. It is defined in such a way that the initial variety $\init_{w}(Y)$ depends only on the cone $\sigma$ containing $w$ in its relative interior. 

When $w \in \Trop(Y)(\ZZ)$ has all negative hyperplane valuations, and $(V,\cA)$ is the complexification of a real reflection arrangement, the fundamental group $\pi_1(\init_{w}(Y)^{\an},c)$ is identified with $\pi_{1}(\Adm(w)^{\an}, x)$ (for any $x$ lying over $c$), which in the terminology of \cite{DRT} is a \emph{pure local wild mapping class group}. 

This gives an alternate definition of such groups, which is intrinsic to the root hyperplane arrangement, and it depends only on a cone in the polyhedral structure on $\Trop(Y)$ induced from Gröbner degenerations. In fact, the intrinsic underlying object is the infinite-dimensional valuation stratum $L_wY$.

Moreover, when $(V,\cA)$ is the type $A$ reflection arrangement, it is well-known that its tropicalization coincides with the space of equidistant phylogenetic trees, cf.~\cite[\S 4.3]{MS}. Integral points with only negative hyperplane valuations correspond exactly to fission trees as defined in \cite{Boalch1}. The tropical framework thus intrinsically reconstructs known invariants of irregular classes.
\end{remark}

\subsection{Standard loops in $\init_{w}(Y)$}

Let $w\in \Trop(Y)(\ZZ) = X_*(T) \cap \Trop(Y)$. Following Theorem~\ref{thm:initialloop}, we further describe the loop 
\[\Pi_{w} : S^1 \to \init_{w}(Y)^{\an} , \ \th \mapsto (e^{i\th})^{w}c= ( e^{w_H i \th} c_H)_{H\in \cA},\]
for some $c\in \init_{w}(Y)^{\an}$. Recall that $w$ determines a hyperplane valuation function $d : \cA \to \ZZ$, and write $\{d(H) \mid H\in \cA\} = \{d_1,...,d_r\}$. Then 
\[ \gr_{w} V = \bigoplus_{k = 1}^{r} V_{d_{k}} / V_{d_{k}-1}, \]
as all other graded pieces vanish. Let 
\[M(k) := M(V_{d_k} / V_{d_{k}-1}, \ov{\cA}_{d_{k}} ) \]
where $\ov{\cA}_{j}$ is the arrangement consisting of all images of hyperplanes $V_j\cap H \in \cA_{j} = \{V_j\cap H \mid H\in \cA, d(H)=j \}$ in $V_{j} / V_{j-1}$. Then
\[ \gr_{w}M = \prod_{k=1}^{r} M(k). \]
The following result now follows from unraveling the definitions.
\begin{cor}\label{cor:relFTfactor}
Let $(c_j) \in \gr_w(M)$ be the image of $c\in\init_{w}(Y)$ under the isomorphism $\Phi_{w} : \gr_{w}M \cong \init_{w}(Y)^{\an}$. Then, under this isomorphism, the loop 
\[\Pi_{w} : S^1 \to \init_{w}(Y)^{\an} , \ \th \mapsto (e^{i\th})^{w}c\]
is identified with the loop 
\[ 
S^1 \to \prod_{k=1}^{r} M(k), \ \th \mapsto (e^{i d_k \th} c_k)_k.
\]
In particular, since $\gr_{w}M$ is a product, after choosing a base-direction on $S^1$, this gives a factorization 
\[ \Pi_{w} = \Pi_{1}^{d_{1}} \cdot \dots \cdot \Pi_{r}^{d_r} \]
into commuting factors, and $\Pi_{k}$ is the central scaling loop in $M(k)$.
\end{cor}

\subsection{Intersection subgroups and (relative) full twists} \label{ss:intersection-sub}

In this subsection, we identify the transport of the individual factors $\Pi_j$ appearing in \ref{cor:relFTfactor} to the nearby fiber $M$ in the Gröbner degeneration. They are described as relative full twists associated to a filtration of $\pi_{1}(M^{\an})$ by so-called intersection subgroups.  In case the arrangement is a complex reflection arrangement, intersection subgroups agree with parabolic subgroups as described in \cite[\S 2 D]{BMR}.

Recall that $V$ is a complex vector space, $\cA$ a finite essential linear hyperplane arrangement, and $M=M(\cA)$ is the corresponding hyperplane complement, which we identify with $Y \subset T$ under the closed immersion into the intrinsic torus. Choose some base-point $y_0 \in M(\cA)$, and denote by $P(\cA) := \pi_{1}(M(\cA),y_0)$ its fundamental group. 

The hyperplane valuations are naturally related to the abelianization of $P(\cA)$. Every $\a_{H}$ induces a map $\a_{H,*} : P(\cA) \to \pi_{1}(\CC^{\times}) \cong \ZZ$. For a loop $b : S^1 \to \cM$, the image $\alpha_{H,*}(b)$ is the homotopy class of the loop $\alpha_{H} \circ b : S^1 \to \CC^{\times}$. Under the identification $\pi_{1}(\CC^{\times}) \cong \ZZ$ its image is the winding number of the loop $\alpha_{H,*}(b)$. Our convention is that the counter-clockwise loop around $0$ has winding number $1$. 
\begin{defn} For \(b\in\pi_1(M,y_0)\), we call \(\alpha_{H,*}(b)\) its winding number around \(H\).
\end{defn}
The following is a reformulation of \cite[Proposition 2.2]{BMR}. 
\begin{prop}
The winding-number pairing
\[\langle-,-\rangle: \pi_1(M,y_0)^{\ab}\times X^*(T) \longrightarrow \mathbb Z,\qquad \langle b,\chi\rangle=\chi_{*}(b)\in \ZZ,\]
is perfect. Equivalently, the induced map
\[\pi_1(M,y_0)^{\ab}\longrightarrow X_*(T) \] 
is an isomorphism.
\end{prop}

Note that this is well-defined on conjugacy classes, so the following Lemma makes sense. 
\begin{lemma} For $f \in Y(F)$, we have $\alpha_{H,*}(b_f) = d_f(H)$. That is, the hyperplane valuation of $f$ for $H$ agrees with the winding number of $b_f$ about $H$. Therefore the image of $b_f$ in $\pi_{1}(\cM,y_0)^{\ab}$ agrees with the hyperplane valuation function $d_{f} \in \ZZ^{\cA}\cong X_*(T)$.
\end{lemma}
\begin{proof}
This follows from the general Proposition \ref{prop:tropwinding} for very affine varieties. 
\end{proof}
\begin{cor} Under the identification $\pi_1(M,y_0)^{\ab}\cong X_*(T)$ the image of the winding number map $Y(F) \to X_*(T)$ coincides with the integral points of $\Trop(Y)(\ZZ)$.
\end{cor}
In the remainder of this section, we will single out a canonical lift (up to conjugacy) in $\pi_{1}(M,y_0)$ of any tropical point $d_f \in \Trop(Y)$. For this purpose, we introduce the full twist loop.

\begin{defn} The loop 
\[ \pi :  [0,1] \to M, s \mapsto y_0 e^{2\pi i s}. \]
is called the \emph{full-twist} with respect to $V$ and $\cA$.
\end{defn}

\begin{remark} The full twist in a linear hyperplane complement is always central by \cite[Lemma 2.4]{BMR}.  In particular, each loop $\Pi_k$ in Corollary \ref{cor:relFTfactor} is exactly the full twist in the arrangement complement $M(k)$ and it is central in $\pi_{1}(M(k)^{\an})$.
\end{remark}

To describe the filtration by intersection subgroups, recall that the intersection lattice $L(\cA)$ is the set of all non-empty intersections of subfamilies of $\cA$. It is ordered by reverse inclusion, and in the linear case has a unique minimal element $Z =  \bigcap_{H \in \cA} H$, sometimes called the center. Let $X\in L(\cA)$ and let $\cA_{X} = \{H\in \cA \mid X \subset H\}$. The rank of $\cA$ is defined to be $\rk(\cA) = \codim_{V} Z$. Note that the arrangement $\cA$ is essential if $\rk(\cA) = \dim(V)$, equivalently if $Z = 0$. Further, if $X\neq Z$, then $\rk(\cA_{X}) < \rk(\cA)$.

In \cite{Par}, the author defines so-called intersection subgroups of $P$. More precisely, the group $\pi_{1}(M_{V}(\cA_{X}),z_{X})$ for some choice of base-point $z_{X}$ is embedded in $P$ as an intersection subgroup of type $X$.

\begin{defn} The image of the full-twist in $P_X := \pi_{1}(M_{V}(\cA_{X}),z_{X})$ in $\pi_{1}(\cM(\cA), y_0)$ is called the full twist of the intersection subgroup $P_X \subset \pi_{1}(\cM(\cA), y_0)$. We denote it by $\pi_{P_X}$. The \emph{relative full twist} in $P$ with respect to $P_X$ is the loop $\pi_{P/P_X} := \pi_{P}\pi_{P_X}^{-1}$.
\end{defn}

For $X\in L(\cA)$ we call $\cA^{X} = \{ H\cap X \mid H \in \cA \setminus \cA_{X} \}$ the restricted arrangement in $X$. It is a finite linear arrangement in $X$. The following is inspired by \cite[\S 4]{Par}.

\begin{lemma}\label{lem:relFTloop} Choose a splitting $V \cong X \oplus V/X$. Let $x \in M_X(\cA^{X})$, and $y\in \cM_{V/X}(\cA_X)$. For $\e > 0$ small enough, the loop 
\[r : S^1 \to M_V(\cA), \qquad e^{i\theta} \mapsto e^{i\theta}x + \e y \]
is a representative for the relative full twist $\pi_{P/P_X}$.
\end{lemma}
\begin{proof} This follows from the following observation: consider the conjugacy class $c_r$ of $r$ in $P$. Then it is easy to check that $\pi_{P}^{-1}c_r$ is represented by the loop $e^{i\theta} \mapsto x+e^{-i\theta} \e y$. But this is a representative of $\pi_{P_X}^{-1}$. Therefore, $r$ represents $\pi_{P}\pi_{P_X}^{-1} = \pi_{P/P_X}$.
\end{proof}

Let $f\in M(F)$ with hyperplane valuation function $d_{f} : \cA \to \ZZ$. Then for each $j$ we have 
\[V_j =   \bigcap_{H \in \cA_{>j}} H \in L(\cA)\]
and this satisfies $\cA_{>j} = \cA_{V_{j}} = \{H \in \cA \mid V_{j}\subset H\} $. In particular, the hyperplane valuation function induces a descending filtration $P_{j}$ of $P$ by intersection subgroups.

Indeed, we may write $f = \sum x_{j} t^{j}$. Then $\cA_{> j} = \{H\in \cA \mid \alpha_{H}(x_{i}) = 0 \textup{ for all} i \le j \}$. Let 
\[V_{j} = \bigcap_{x_{i}\in H\textup{ for all} i \le j} H = \bigcap_{H \in \cA_{>j}} H.\]
By definition, one always has that $\cA_{>j} \subset \cA_{V_{j}}$. The other direction follows because $\cA_{>j}$ comes from the loop $f$: Let $H \in \cA_{V_{j}}$, i.e. $V_j \subset H$. Then we need to check that $d_{f}(H) > j$, which is equivalent to $x_i \in H$ for all $i \le j$. However, $V_j$ contains every such $x_i$. So $x_i \in V_j \subset H$ for all $i \le j$ implies the claim.

\subsection*{Notation} Let $f\in M(F)$. Then we have $\cA_{>j} = \emptyset$ for large $j$. We write
\[ \{d_{f}(H) \mid H\in \cA \} = \{d_1,...,d_r \}, \]
ordered as $d_1 < d_2 < \dots < d_r$. 

For $k=1,...,r$ let $P_{k} \subset P$ denote the intersection subgroup of type $X_{d_{k}}$. Note that $P_{r}$ is trivial, so the relative full twist $\pi_{P_{r-1}/P_{r}} = \pi_{P_{r-1}}$ is the full twist for the intersection subgroup $P_{r-1}$.

\begin{theorem}\label{thm:braidrep}
 For $f\in M(F)$ with corresponding filtration $(P_{k})$ as above, the conjugacy class of $\g_f=\cL_{\alg}(f)$ in $P$ is represented by the loop
\[ \pi_{P/P_1}^{d_1} \pi_{P_1/P_2}^{d_2} \dots \pi_{P_{r-2}/P_{r-1}}^{d_{r-1}}\pi_{P_{r-1}}^{d_{r}}.\]
\end{theorem}

\begin{proof}
By Corollary \ref{cor:relFTfactor}, the loop $\g_f$ has a factorization
\[
\Pi_1^{d_1}\cdots\Pi_r^{d_r}
\] 
into pairwise commuting loops obtained by transporting the scaling loops of the graded factors of $\gr_wM$. It therefore suffices to show that
\[
\Pi_k=\pi_{P_{k-1}/P_k},
\qquad P_0:=P.
\]

Choose a splitting
\[
V=W_1\oplus\cdots\oplus W_r
\]
of the valuation filtration. For
$u=(u_1,\ldots,u_r)\in\gr_wM=\prod_jM(j)$, set
\[
\widetilde u(z) = \sum_{j=1}^r z^{d_j}u_j.
\]
Let
\[
c_f=(c_1,\ldots,c_r)
\]
and consider the $k$-th scaling loop
\[
c_{k,\theta} = (c_1,\ldots,c_{k-1},e^{i\theta}c_k,c_{k+1},\ldots,c_r) 
\]
in the special fiber.

If $H\in\cA$ has valuation level $d_j$, then
\[
z^{-d_j}\alpha_H(\widetilde c_{k,\theta}(z)) = \operatorname{in}_w(\alpha_H)(c_{k,\theta})+O(z),
\]
uniformly in $\theta$. Moreover,
\[
\left| \operatorname{in}_w(\alpha_H)(c_{k,\theta}) \right| = |\alpha_H(c_j)|>0.
\]
Since $\cA$ is finite, there is therefore a $\delta>0$ such that
\[
\widetilde c_{k,\theta}(z)\in M
\]
for every $\theta$ and every $0<|z|<\delta$.

Consequently, the weighted lift defines an admissible transport of the $k$-th scaling loop from the special fiber to any sufficiently small nearby fiber. By Corollary \ref{cor:centloop}, for $0<\e<\delta$ the transported loop is represented in $M$ by
\[
\pi_k(\theta) = \sum_{j<k}\e^{d_j}c_j + \e^{d_k}e^{i\theta}c_k + \sum_{j>k}\e^{d_j}c_j.
\]

Multiplication by the positive scalar $\e^{-d_1}$ does not change its homotopy class. Thus we may write
\[
\pi_k(\theta) = c_+ + \e^{d_k-d_1}e^{i\theta}c_k + c_-,
\]
where
\[
c_+ = \sum_{j<k}\e^{d_j-d_1}c_j, \qquad c_- = \sum_{j>k}\e^{d_j-d_1}c_j.
\]
Since $d_j - d_1 > 0$ for all $j$, when $\e$ is sufficiently small this loop lies in the local arrangement complement corresponding to $P_{k-1}$. Applying Lemma \ref{lem:relFTloop} to this local arrangement therefore shows
\[
\Pi_k=\pi_{P_{k-1}/P_k}.
\]
Substituting this into the factorization above gives
\[
\g_f
=
\pi_{P/P_1}^{d_1}
\pi_{P_1/P_2}^{d_2}
\cdots
\pi_{P_{r-2}/P_{r-1}}^{d_{r-1}}
\pi_{P_{r-1}}^{d_r}
\]
up to conjugacy.
\end{proof}

\begin{remark}
In the situation of the preceding theorem, one can also compute $\g_f$ using a suitable tropical compactification. For an essential linear hyperplane arrangement, one may use a refinement of the wonderful compactification of De Concini and Procesi. Since $\Trop(Y)$ contains the lineality space $\RR(1,\ldots,1)$, we subdivide
\[
\RR=\RR_{\geq 0}\cup\RR_{\leq 0},
\]
so that the two lineality directions become rays of the refined fan. The resulting fan is smooth, and every $w_f$ is contained in the relative interior of a unique cone $\s_f$.

Let $d_1<d_2<\cdots<d_r$ be the distinct values of the valuation function $w_f$ on $\cA$, and set
\[
\cA_{>d_k}:=\{H\in\cA\mid w_f(H)>d_k\},
\]
for $1\leq k\leq r-1$. Let $\rho_k$ denote the ray whose primitive generator is the indicator function of $\cA_{>d_k}$. Furthermore, let
\[
\rho_{+}:=\RR_{\geq0}\mathbf 1_{\cA},
\qquad
\rho_{-}:=\RR_{\geq0}(-\mathbf 1_{\cA})
\]
be the two lineality rays, with primitive generators
\[
v_{\rho_{+}}=\mathbf 1_{\cA},
\qquad
v_{\rho_{-}}=-\mathbf 1_{\cA}.
\]
If $d_1\neq 0$, let $\epsilon=\operatorname{sgn}(d_1)\in\{+,-\}$ and $a_0=|d_1|$. Then
\[
w_f=a_0v_{\rho_{\epsilon}}+\sum_{k=1}^{r-1}(d_{k+1}-d_k)v_{\rho_k},
\]
while for $d_1=0$ the first summand is omitted. Thus this is precisely the decomposition of $w_f$ into the primitive ray generators of $\s_f$.

Let $P$ denote the full intersection subgroup and let $P_k$ be the intersection subgroup corresponding to $\cA_{>d_k}$. One checks that the two lineality rays correspond to the full twists
\[
\cL_{\trop}(v_{\rho_{+}})=\pi_P,
\qquad
\cL_{\trop}(v_{\rho_{-}})=\pi_P^{-1},
\]
while the ray $\rho_k$ corresponds to the full twist $\pi_{P_k}$. Hence by the ray formula in Proposition \ref{p:rayformula}, 
\[
\g_f=\pi_P^{\,d_1}\pi_{P_1}^{\,d_2-d_1}\pi_{P_2}^{\,d_3-d_2}\cdots\pi_{P_{r-1}}^{\,d_r-d_{r-1}}.
\]
Since $d_1<d_2<\cdots<d_r$, all exponents except possibly $d_1$ are positive. Equality with
\[
\pi_{P/P_1}^{d_1}\pi_{P_1/P_2}^{d_2}\cdots
\pi_{P_{r-2}/P_{r-1}}^{d_{r-1}}\pi_{P_{r-1}}^{d_r}
\]
also follows directly from the algebraic identities
\[
\pi_{P/P_1}=\pi_P\pi_{P_1}^{-1},
\qquad
\pi_{P_i/P_{i+1}}=\pi_{P_i}\pi_{P_{i+1}}^{-1},
\]
by telescoping.
\end{remark}

\section{Cohomology of braid varieties} \label{s:cohom-braid-var}
In this section, we apply the preceding results on algebraic loops to braid varieties associated with Weyl groups. A formal loop in the regular semisimple locus determines a braid, and the root valuations of the loop give an explicit positive representative as a product of relative full twists. We then use the horocycle correspondence and its cyclicity property to show that the ordinary and compactly supported cohomology of the corresponding braid varieties depends only on the conjugacy class of this braid. Combining these facts with the tropical classification of algebraic loops, we conclude that, for positive pure
algebraic braids, the cohomology of the associated braid varieties is controlled entirely by the root valuation data.

We start by fixing some notation. 
 Let $G$ be a simple complex algebraic group with Lie algebra $\frg$, and
fix a maximal torus $T\subset G$ with Cartan subalgebra
$\frt=\Lie(T)$. Let $\Phi\subset\frt^*$ be the corresponding root
system and let $W$ be its Weyl group. The root hyperplanes
$\ker(\alpha)\subset\frt$ form an essential hyperplane arrangement
$\cA$, naturally indexed by $\Phi/\{\pm1\}$, and its complement is the
regular semisimple locus
\[
Y=\frt^{\rs}:=
\frt\setminus\bigcup_{\alpha\in\Phi}\ker(\alpha).
\]
Let $\frc=\frt/\!/W$ be the Chevalley quotient. Since $W$ acts freely on
$\frt^{\rs}$, its regular locus is
\[
\frc^{\rs}=\frt^{\rs}/W.
\]
Consequently,
\[
P_W:=\pi_1\bigl((\frt^{\rs})^{\an}\bigr)
\qquad\text{and}\qquad
B_W:=\pi_1\bigl((\frc^{\rs})^{\an}\bigr)
\]
are respectively the pure braid group and the Artin braid group
associated with $W$, and they fit into the exact sequence
\[
1\longrightarrow P_W\longrightarrow B_W\longrightarrow W
\longrightarrow1.
\]

\subsection{Relative full twist} 

The explicit representative for an algebraic loop obtained in Theorem~\ref{thm:braidrep} is expressed in terms of relative full twists associated with nested parabolic subgroups. For the applications to braid varieties below, it is important to represent these elements by positive braids. In this subsection, we give a Coxeter-theoretic construction of relative full twists using positive lifts of longest elements. We show that this construction agrees with the earlier geometric definition, behaves naturally along chains of parabolic subgroups, and yields an explicit positive representative for the braid associated with a root valuation function.

 For any parabolic subgroup $W_H \subset W$ we denote by $w_{0}^{H}$ the longest element of $W_{H}$, and we let $w_0$ be the longest element of $W$.

\begin{lemma}
Let $x=w_{0}^{H}w_0$ and $y=w_0w_{0}^{H}$. Then $\ell(w_0) = \ell(w_{0}^{H})+\ell(x) = \ell(y)+\ell(w_{0}^{H})$.
\end{lemma}
\begin{proof}
This follows from \cite[Lemma 1.5.3]{GP}.
\end{proof}

\begin{defn}
    We write $x=w_{H \backslash G}$ and $y=w_{G/H}$, and denote their braid lifts by $\Delta_{H \backslash G}$ and $\Delta_{G/H}$. We call $\pi_{G/H} = \Delta_{H \backslash G}\Delta_{G/H}$ the \textit{relative full-twist} with respect to $W_H \subset W$. It is a pure braid generalizing the full-twist (the case $W_H=1$). Moreover we let $\Delta_{H}$ be the positive braid lift of $w_{0}^{H}$, and $\pi_{H} = \Delta_{H}^{2}$. When $H=G$ we sometimes write $\Delta_G = \Delta$.
\end{defn}
In the following we prove a few elementary properties of relative full-twists.

\begin{lemma}
    We have $\Delta_{H} \Delta_{H \backslash G} = \Delta_{G/H}\Delta_{H} = \Delta_{G}$.
\end{lemma}
\begin{proof}
This follows from the length additivity above, since positive braid lifts respect reduced products.
\end{proof}

\begin{lemma}
    The relative full-twist $\pi_{G/H}$ commutes with $\Delta_{H}$. Moreover, $\pi_{G/H}\pi_{H} = \pi_{G}$. In particular, the full twist defined from positive representatives agrees with the one defined previously. 
\end{lemma}
\begin{proof}
First note that 
\[\Delta_{H \backslash G} \Delta_{G/H} \Delta_{H} = \Delta_{H \backslash G} \Delta_{G}\]
and
\[ \Delta_{H} \Delta_{H \backslash G} \Delta_{G/H} = \Delta_{G} \Delta_{G/H}.\]
We thus need to show $\Delta_{G} \Delta_{G/H} = \Delta_{H \backslash G} \Delta_{G}$. After multiplying by $\Delta_{G}$ on the right, we can use the fact that $\pi_{G}$ is central in the braid group. So using that $\Delta_{G} = \Delta_{G/H}\Delta_{H}$, we conclude
\[\Delta_{H \backslash G} \Delta_{G}^2 = \Delta_{G}^{2} \Delta_{H \backslash G} = \Delta_{G}\Delta_{G/H} \Delta_{H} \Delta_{H\backslash G} = \Delta_{G} \Delta_{G/H}\Delta_{G}. \]
Dividing both sides on the right by $\Delta_{G}$ proves the first claim.

To prove the second statement, we note that we can use the first claim to rewrite
\[\pi_{G/H}\pi_{H} = \Delta_{H\backslash G} \Delta_{G/H} \Delta_{H}^{2} = \Delta_{H} \Delta_{H\backslash G} \Delta_{G/H} \Delta_{H}.\]
By the definition of $\Delta_{G/H}$ and $\Delta_{H \backslash G}$, this is equal to $\Delta_{G}^2 = \pi_{G}$.
\end{proof}
\begin{cor}
    Let $W_K \subset W_H \subset W$ be parabolic subgroups. Then
    \begin{enumerate}
        \item $\pi_{G/K} = \pi_{G/H}\pi_{H/K}$.
        \item The relative full twist $\pi_{G/H}$ commutes with any braid in the parabolic subgroup $W_{H}$.
    \end{enumerate}
\end{cor}
\begin{proof}
Both statements follow easily from $\pi_{G/H} = \pi_{G}\pi_{H}^{-1}$ for a parabolic subgroup $H$.
\end{proof}

We briefly recall the definition of the root valuation datum associated with a point
$f \in \mathfrak c^{\mathrm{rs}}(F)$ following \cite{GKM}.  After a finite ramified base change
$t = u^m$, choose a lift $\widetilde f \in \mathfrak t^{\mathrm{rs}}(\mathbb C((u)))$. For each root $\alpha \in \Phi$, set
\[d_{\widetilde f}(\alpha) = \operatorname{val}_u\langle \alpha,\widetilde f\rangle / m = \val_{t} \langle \alpha,\widetilde f\rangle.\]
The lift is well-defined only up to the Weyl group and the Galois action
$u \mapsto \zeta u$; the associated root valuation datum is the corresponding equivalence class $[d,w]$ with $W$ acting on functions $d$ by $(x.d)(\a) = d(x^{-1}(\a))$ and by conjugation on $w$. Specializing Theorem \ref{thm:braidrep} to braid groups, we obtain the following result. 
\begin{theorem}\label{thm:braidprod}
Let $f\in\frt^{\rs}(F)$ with root valuations
\[ \{d_f(\alpha) \mid \alpha \in \Phi \} = \{d_1,...,d_r \}, \]
and corresponding filtration 
\[W\supset W_1 \supset W_2 \supset \cdots \supset W_{r-1} \supset \{1\}, \] 
then the conjugacy class $\g_f = \cL_{\alg}(f) $ in $P_{W}$ contains the canonical representative
\[\pi_{W/W_1}^{d_1} \pi_{W_1/W_2}^{d_2} \dots \pi_{W_{r-2}/W_{r-1}}^{d_{r-1}}\pi_{W_{r-1}}^{d_{r}}.\]
When $f \in \frc^{\rs}(F)$ has root valuation datum $[d_f, w]$ with $\ord(w) = m$, then we can write
\[ \{d_f(\alpha) \mid \alpha \in \Phi \} = \{d_1/m,...,d_r/m \}, \]
for integers $d_1,...,d_r$ and $\g_f^m$ is conjugate to 
\[\pi_{W/W_1}^{d_1} \pi_{W_1/W_2}^{d_2} \dots \pi_{W_{r-2}/W_{r-1}}^{d_{r-1}}\pi_{W_{r-1}}^{d_{r}}.\]
\end{theorem}
\begin{proof}
The only thing that remains to be shown is the second statement. It remains only to check that $f(t^{m}) \in \frt^{\rs}(F)$. The loop attached to this computes an $m$-th power of $\g_f$, and applying Theorem \ref{thm:braidrep} to it shows that its $m$-th power has the desired shape.
\end{proof}
\begin{remark} We expect that for $f\in \frc^{\rs}(F)$ with root valuation datum having non-trivial $w\neq 1$ (the twisted or non-pure case), the corresponding braid can be more explicitly described as a good position braid representative in the sense of Duan \cite{Duan}. This non-pure case will be studied in more detail in a future paper. \end{remark}

\subsection{Root valuation data determine the braid}

Even though we cannot give an explicit formula yet in the twisted case, we can still prove directly that the root valuation datum determines the conjugacy class of the associated algebraic braid. 

Let $f\in\frc^{\rs}(F)$ have root valuation datum $[d,w]$, put
$m=\ord(w)$, and let $\zeta$ be a primitive $m$-th root of unity.
After the base change $t=u^m$, choose a lift
\[
\widetilde f\in\frt^{\rs}(\CC(\!(u)\!))
\]
satisfying $\widetilde f(\zeta u)=w\widetilde f(u)$. Writing
\[
\widetilde f(u)=\sum_{k\geq k_0}v_ku^k,
\]
the equivariance condition is equivalent to
\[
wv_k=\zeta^kv_k.
\]
We therefore set
\[
\frt_k(w):=\ker(w-\zeta^k)\subset\frt,
\]
so that $v_k\in\frt_k(w)$. For every $\a\in\Phi$, write
\[
q_\a:=\val_u\langle\a,\widetilde f\rangle
=m d(\a)\in\ZZ.
\]

\begin{prop} \label{prop:root-valuations-determine-braid}
Let $f_0,f_1\in\frc^{\rs}(F)$ have the same root valuation datum.
Then  $\g_{f_0}=\cL_{\alg}(f_0)$ and $\g_{f_1}=\cL_{\alg}(f_1)$ are homotopic. In other words, the conjugacy class in $B_W$ of the braid associated with $f\in\frc^{\rs}(F)$ depends only on its root valuation datum.
\end{prop}
\begin{proof}
Choose a representative $(d,w)$ of the common root valuation datum.  Put
\[
m:=\ord(w),
\]
choose a primitive \(m\)-th root of unity \(\zeta\), and, after the
base change \(t=u^m\), choose lifts
\[
\widetilde f_i(u)
=
\sum_{k\geq k_0}v_{i,k}u^k
\in\frt^{\rs}(\CC(\!(u)\!)),
\qquad i=0,1,
\]
satisfying
\[
\widetilde f_i(\zeta u)=w\widetilde f_i(u).
\]
In particular, $v_{i,k}\in\frt_k(w)$. For $\alpha\in\Phi$, write $q_\alpha:=md(\alpha)\in\ZZ$, 
and set $k_0:=\min_{\alpha\in\Phi}q_\alpha$.  Since the roots span $\frt^*$, every lift with root valuation function $d$ has an expansion
\[
\widetilde f_i(u)=\sum_{k\geq k_0}v_{i,k}u^k.
\]
  We first truncate without changing the algebraic loop, and then connect the resulting Laurent polynomials inside a connected twisted admissible deformation space. For that, choose \(N\) such that $N\geq\max_{\a\in\Phi}q_\a$, and let
\[
\widetilde f_i^{[N]}(u)
:=
\sum_{k=k_0}^{N}v_{i,k}u^k
\]
be the corresponding Laurent-polynomial truncation.  The truncation remains equivariant, since
\[
\begin{aligned}
\widetilde f_i^{[N]}(\zeta u)
&=
\sum_{k=k_0}^{N}\zeta^k v_{i,k}u^k\\
&=
\sum_{k=k_0}^{N}wv_{i,k}u^k
=
w\widetilde f_i^{[N]}(u).
\end{aligned}
\]
It also preserves all root valuations. Indeed, for every $\a\in\Phi$, the
series
\(\langle\a,\widetilde f_i^{[N]}\rangle\) contains the leading
nonzero term of
\(\langle\a,\widetilde f_i\rangle\), because \(N\geq q_\a\).

We claim that replacing \(\widetilde f_i\) by
\(\widetilde f_i^{[N]}\) does not change the associated algebraic
loop.  Consider the family
\[
\widetilde F_i(s,u)
:=
\widetilde f_i^{[N]}(u)
+s\bigl(
\widetilde f_i(u)-\widetilde f_i^{[N]}(u)
\bigr).
\]
For every root \(\a\in\Phi\), since $N \ge q_{\a}$, we have
\[
\langle\a,\widetilde F_i(s,u)\rangle
=
u^{q_\a}
\bigl(c_{i,\a}+uR_{i,\a}(s,u)\bigr),
\]
for some $c_{i,\a}\in\CC^\times$, and some $R_{i,\a}(s,u)\in \CC[s][[u]]$. In particular, its leading coefficient is non-zero and independent of $s$, so $\langle\a,\widetilde F_i(s,u)\rangle$ is a unit in $\CC[s](\!(u)\!)$. Therefore
\[
\widetilde F_i
\in
\frt^{\rs}\bigl(\CC[s](\!(u)\!)\bigr).
\]
Moreover,
\[
\widetilde F_i(s,\zeta u)
=
w\widetilde F_i(s,u).
\]
After applying the quotient map
\[
p\colon\frt^{\rs}\longrightarrow\frc^{\rs},
\]
the family therefore descends along \(t=u^m\) to an \(\AA^1\)-family
\[
F_i\in
\frc^{\rs}\bigl(\CC[s](\!(t)\!)\bigr)
=
L(\frc^{\rs})(\AA^1).
\]
Its fibers at \(s=0\) and \(s=1\) are respectively
\[
f_i^{[N]}
:=
p\bigl(\widetilde f_i^{[N]}\bigr)
\qquad\text{and}\qquad
f_i.
\]
Since \(\AA^1\) is connected, the constancy of the algebraic-loop map in connected algebraic families, Corollary~\ref{cor:constancy-components}, gives
\[
\g_{f_i^{[N]}}=\g_{f_i}.
\]
We may consequently replace $f_0,f_1$ by $f_0^{[N]},f_1^{[N]}$ and hence assume that their lifts are Laurent polynomials of degree at most $N$.

For a Laurent polynomial
\[
\widetilde g(u) = \sum_{k=k_0}^{N}v_ku^k, \qquad v_k\in\frt_k(w),
\]
the condition
\[
\ord_u\langle\alpha,\widetilde g\rangle=q_\alpha
\]
is equivalent to
\[
\langle\alpha,v_k\rangle=0 \quad\text{for }k<q_\alpha, \qquad \langle\alpha,v_{q_\alpha}\rangle\neq0.
\]
Thus, if
\[
\cV_N(d,w) = \left\{ (v_{k_0},\ldots,v_N) \in \bigoplus_{k=k_0}^{N}\frt_k(w) \ \middle|\ \langle\alpha,v_k\rangle=0 \text{ for }k<q_\alpha \right\},
\]
then the coefficient tuples defining loops with root valuation datum $(d,w)$ form the open subset
\[
\cJ_N(d,w) = \cV_N(d,w) \setminus \bigcup_{\alpha\in\Phi} \left\{ \langle\alpha,v_{q_\alpha}\rangle=0 \right\}.
\]
Since $(d,w)$ is realized, each hyperplane removed above is proper.
Hence $\cJ_N(d,w)$ is a nonempty complement of finitely many complex
linear hyperplanes and is therefore connected.

There is a universal equivariant Laurent polynomial
\[
\widetilde{\mathcal F}(u)
=
\sum_{k=k_0}^{N}v_ku^k
\]
over $\cJ_N(d,w)$. By construction, $ \widetilde{\mathcal F}(\zeta u) = w\widetilde{\mathcal F}(u)$
and $\val_u \langle\alpha,\widetilde{\mathcal F}\rangle = q_\alpha$ for every $\alpha\in\Phi$. Hence, after applying the quotient map
$p:\frt^{\rs}\to\frc^{\rs}$ and descending along $t=u^m$, it defines
an algebraic family
\[
\mathcal F\in L(\frc^{\rs})\bigl(\cJ_N(d,w)\bigr)
\]
whose fibers include $f_0^{[N]}$ and $f_1^{[N]}$. Since $\cJ_N(d,w)$ is connected, constancy of the algebraic loop map
in connected algebraic families gives
\[
\g_{f_0^{[N]}} = \g_{f_1^{[N]}}, 
\]
proving $\g_{f_0}=\g_{f_1}$.
\end{proof}

\begin{remark} The same result holds in the more general situation of a central finite hyperplane arrangement $\cA$ in a complex vector space $V$. Assume there is a finite group $\G$ acting linearly on $V$, preserving $\cA$ and such that the action on the arrangement complement is free, and denote the quotient by $X$. 

For \(f\in X(F)\), after a ramified base change \(t=u^m\), one may
choose a lift
\[
\widetilde f\in M(\mathcal A)(\CC(\!(u)\!))
\]
satisfying
\[
\widetilde f(\zeta u)=g\widetilde f(u)
\]
for some \(g\in\G\) of order \(m\), where \(\zeta\) is a primitive \(m\)-th root of unity.  Choosing defining linear forms \(\ell_H\) for \(H\in\mathcal A\), set
\[
d_f(H)
=
\frac{1}{m}\ord_u\ell_H(\widetilde f).
\]
Changing the lift by an element of \(\G\) changes the pair \((d_f,g)\) by the natural simultaneous \(\G\)-action. We call its \(\G\)-orbit the hyperplane valuation datum of \(f\).

The same proof as that of Proposition \ref{prop:root-valuations-determine-braid} shows that the loop $\g_f$ depends only on the hyperplane valuation datum of $f$. 
\end{remark}

\begin{remark}
The spaces $\cJ_N(d,w)$ have already appeared in \cite{GKM}, and under the name of twisted admissible deformation spaces in \cite{DRY}.
\end{remark}

\subsection{Horocycle correspondence and cohomology of braid varieties}

In this subsection, we explain how to derive consequences on the cohomology of certain braid varieties. Let $\cB$ be the variety of Borel subgroups of $G$ and $N$ its dimension. Then, the orbits for the diagonal $G$-action on $\cB \times \cB$ are in bijection with $W$. For $w \in W$, denote the corresponding orbit by $\bO(w)$. Two Borel subgroups $B_1,B_2$ are said to be in relative position $w$ if $(B_1,B_2)\in \bO(w)$. For a positive braid $\b = \tilde{w}_1...\tilde{w}_n$ we define
\[ Z(\beta) = \{ (B_0, ..., B_{n}, g ) \in \cB^{n+1} \times G \mid (B_{i},B_{i+1}) \in \bO(w_{i+1}), {}^{g}B_0 = B_n \}. \]
There is a natural projection map $Z(\beta) \to G$. We denote the fiber over $g \in G$ by $X(\beta,g)$, and abbreviate $X(\beta) = X(\beta,1)$. For any conjugacy class $C\subset G$ we denote by $X(\beta,C)$ the pre-image of $C$ under $Z(\b) \to G$.

\begin{remark} In type $A$, the cohomology of these spaces is related to knot invariants, and for certain braids, these spaces can be thought of as moduli spaces of Stokes filtered local systems on a disk. \end{remark}

To study their cohomology, we briefly recall the horocycle correspondence, and how the spaces $Z(\b)$ and $X(\b,g)$ are related to it. 

In the following we will work with the full derived category $D^b_G(\cB \times \cB)$ of $G$-equivariant constructible sheaves on the analytification. The horocycle correspondence is the following diagram 
\[
\xymatrix{ &  \cB \times G \ar[dl]_{f} \ar[dr]^{\varpi} &  \\
\cB \times \cB & & G
}
\]
with maps $f(B,g) = (B,{}^{g}B)$ and $\varpi(B,g) = g$ the projection to $G$. The maps are equivariant for the diagonal $G$-action on $\cB \times \cB$, the action $h.(B,g) = ({}^{h}B, hgh^{-1})$ on $\cB \times G$ and the conjugation action on $G$. 

Consider the character functor 
\begin{align*}
\chi : D^b_{G} (\cB \times \cB) \to D^b_{G}(G) \\
\chi = \varpi_{!}f^{*}.
\end{align*}

Recall that $D^b_G(\cB \times \cB)$ carries a convolution product, induced from the diagram
\[
\xymatrix{ &  \cB\times \cB\times \cB \ar[r]_{p_{13}} \ar[dl]_{p_{12}} \ar[dr]_{p_{23}} & \cB \times \cB \\
\cB \times \cB & & \cB \times \cB.
}
\]
For any two objects $K,L \in D^b_{G}(\cB \times \cB)$ their convolution is defined by
\[K * L = (p_{13})_! (p_{12}^*K \otimes p_{23}^*L )[-N]. \]

The character functor satisfies the following cyclicity property.
\begin{prop}\label{prop:cattrace}
For $K,L \in D^b_{G}(\cB \times \cB)$ there is a canonical isomorphism $\chi(K * L) \cong \chi(L*K)$.
\end{prop}
\begin{proof}
This can be found in \cite[1.11.(a)]{Lus}, noting that all arguments used remain valid in our setting.
\end{proof}
Equivalently, $\chi$ factors through the categorical trace.

\subsection{Standard and costandard objects} We follow the notation from \cite[\S 1]{BDR}. Recall that $N = \dim\cB$. Let $B_W^+$ be the positive braid monoid. For any positive braid $\b = \tilde{w}_1...\tilde{w}_n$ we define
\[ \bO(\beta) =  \{ (B_0, ..., B_{n} ) \in \cB^{n+1} \mid (B_{i},B_{i+1}) \in \bO(w_{i+1}) \}. \]
We have a natural map $j_{\beta} : \bO(\beta) \to \cB \times \cB$ given by $ (B_0, ..., B_{n}) \mapsto (B_0,B_{n})$.
\begin{defn}
The standard and costandard objects in $D^b_{G}(\cB \times \cB)$ are 
\begin{align*}
\Delta(\beta) &= j_{\beta,!}{\QQ}[\ell(\beta)+N], \\
\nabla(\beta) &= j_{\beta,*}{\QQ}[\ell(\beta)+N]. 
\end{align*}
\end{defn}
The standard objects satisfy $\Delta(\beta \beta') = \Delta(\beta) * \Delta(\beta')$ and similarly for the costandard objects.  Moreover, denoting by $\beta \mapsto \beta^{*}$ the anti-involution on $B_{W}^{+}$ which for $\beta = \tilde{s}_{i_{1}}\dots \tilde{s}_{i_{r}}$ satisfies $\beta^{*} = \tilde{s}_{i_{r}}\dots \tilde{s}_{i_{1}}$, we have isomorphisms
\[\Delta(\beta) * \nabla(\beta^{*}) \cong \Delta(1) \cong \nabla(\beta^{*}) *\Delta(\beta).\]
In this way one can define the object $\Delta(\beta)$ for any braid $\beta \in B_{W}$, and such that it satisfies $\Delta(\beta^{-1}) = \nabla(\beta^{*})$ whenever $\beta\in B_{W}^{+}$. 

\begin{lemma}
Let $\beta$ be a positive braid, and $g\in G$. Then we have isomorphisms 
\begin{enumerate}
\item $R\Gamma_{c}(G,\chi(\Delta(\b))) \cong \cohoc{*+\ell(\beta)+N}{Z(\beta),\QQ}$,
\item $R\Gamma(G,\chi(\nabla(\b))) \cong \cohog{*+\ell(\beta)+N}{Z(\beta),\QQ}$,
\item $\chi(\Delta(\b))_{g} \cong \cohoc{*+\ell(\beta)+N}{X(\b,g), \QQ}$, 
\item $\chi(\nabla(\b))_{g} \cong \cohog{*+\ell(\beta)+N}{X(\b,g),\QQ}$,
\end{enumerate}
where the subscript denotes the stalk at $g$.
\end{lemma}
\begin{proof} 
All of these are simple base-change calculations (note $f$ is smooth and $\varpi$ is proper). We suppress the cohomological shifts in this calculation. Recall $\chi(\Delta(\b)) = \varpi_{!}f^{*} j_{\b,!}\QQ$. Consider the diagram
\[
\xymatrix{
Z(\b)\ar[r]^{j_{\b}'}\ar[d]^{f'} & \cB \times G \ar[d]_{f} \ar[dr]^{\varpi} & \\
\bO(\b) \ar[r]^{j_{\b}} & \cB \times \cB & G,
}
\]
where the square is Cartesian (this can be checked easily). Let $\varpi' = \varpi\circ j_{\b}'$. Then applying base-change gives
\[
\chi(\Delta(\b)) = \varpi_{!}f^{*} j_{\b,!}\QQ = \varpi_{!}'\QQ.
\]
Pushing further along the structural map $G \to \pt$ proves (1). Moreover
\[ \chi(\Delta(\b))_{g} = (\varpi_{!}'\QQ)_g = \cohoc{*}{X(\b,g),\QQ}. \]
To get the similar statements for $\nabla(\b)$ one uses smooth-base change along $f$, and one uses that $\varpi$ is proper. Otherwise the calculation is the same.
\end{proof}

Since cohomology of $Z(\b)$ and related spaces is computed from $\chi(\Delta(\b))$ and $\chi(\nabla(\b))$, the cyclicity property of $\chi$ in Proposition \ref{prop:cattrace} implies the following conjugacy-invariance statements for cohomology.

\begin{cor}\label{cor:conjinv}
Assume $\b$ and $\b'$ are two conjugate positive braids. Then there are isomorphisms
\[
\chi(\Delta(\beta))\cong \chi(\Delta(\beta')), 
\qquad
\chi(\nabla(\beta))\cong \chi(\nabla(\beta'))
\]
in $D^b_G(G)$. Consequently, compactly supported and ordinary cohomology of $Z(\beta)$, of the fibers $X(\beta,g)$, and of their equivariant versions over conjugacy classes depend only on the conjugacy class of the positive braid $\beta$.
\end{cor}

\subsection{Valuation data dependence}

There are two cases in which we can naturally attach a positive braid to $f\in \frc^{\rs}(F)$. We recall how to construct this positive representative briefly, following \cite[\S 4.3.2]{BBMY}. We need to introduce a slight generalization of Stokes directions for this. 
 
\begin{defn}
Let $f\in \frc^{\rs}(F)$, and choose a lift
\[
\widetilde f\in \frt^{\rs}(\CC(\!(s)\!))
\]
after the base change $t=s^m$. For each $\alpha\in\Phi$, write
\[
\langle \alpha,\widetilde f\rangle(s) = a_\alpha s^{q_\alpha}+\cdots, \qquad a_\alpha\in\CC^\times, \]
and set $d_\alpha:=\frac{q_\alpha}{m}$. Thus $d_\alpha$ is the root valuation of $\alpha$ with respect to $f$.

For a level $q\in\ZZ$, define the set of critical directions at level $q$ by
\[
\cS_f(q) := \left\{ \theta\in[0,1) \ \middle|\ \exists\,\alpha\in\Phi \text{ with } q_\alpha=q \text{ and } \Re\!\left(a_\alpha e^{2\pi i q_\alpha\theta}\right)=0 \right\}. \]
The set of critical directions of $f$ is
\[ \cS_f := \bigcup_{q\in\ZZ}\cS_f(q). \]
\end{defn}
\begin{remark}
\begin{enumerate} 
\item When $f \in \frc^{\rs}(F) \cap \frc(t^{-1}\CC[t^{-1}])$, then $f$ is essentially an irregular class, and the critical directions of $f$ coincide with the Stokes directions. 
\item The above definition has a generalization for quotients of very affine varieties. Let $Y\subset T$ be very affine with free action of a finite group $G$. For $f\in Y(F)$ let $w=w_f \in \Trop(Y)(\ZZ)$ be the corresponding tropical point and $c_{\widetilde{f}} \in \init_{w}(Y)$ the central point in the Gröbner degeneration, then for every character $\chi \in X^*(T)$ one has 
\[ \chi(\widetilde{f}(t)) = \chi(c_{\widetilde{f}}) t^{\langle \chi,w\rangle} + \cdots \]
We can define critical directions for every $\chi \in X^*(T)$ in this way. These are independent of the choice of lift $\widetilde{f}$. Putting $Y= \frt^{\rs}, G=W$ and picking the coordinate characters on the intrinsic torus $T$ given by roots gives back the previous definition.
\end{enumerate}
\end{remark}

\begin{prop}
Let $f\in \frc^{\rs}(F)$, assume that $d_f(\a) \neq 0$ for all roots
$\a$ (so the set $\cS_f$ is finite), and let $r$ be the number of
critical directions of $f$. Choose a base-direction $\theta_0$ on $S^1$.
Then there is a canonical sequence $(C_1,\ldots,C_r)$ of Weyl chambers
determined by $f$ (up to diagonal $W$-action). Moreover, let
$w_i=\mathrm{relpos}(C_i,C_{i+1})$ for $i=1,\ldots,r-1$, and let
$w_r=\mathrm{relpos}(C_r,C_1)$. Then, if all root valuations of $f$ are
positive (resp. negative), the conjugacy class $\g_f$ (resp.
$\g_f^{-1}$) contains the positive braid
\[
\b_f^+:=\mathbf w_1\cdot\ldots\cdot\mathbf w_{r-1}\mathbf w_r .
\]
\end{prop}

\begin{proof}
Choose a lift
\[
\widetilde f\in \frt^{\rs}(\CC(\!(s)\!))
\]
after the base change $t=s^m$, and write
\[
\langle\a,\widetilde f\rangle(s) = a_\a s^{q_\a}+\cdots .
\]
Let $J_1,\ldots,J_r$ be the connected components of $S^1\setminus\cS_f$, ordered counter-clockwise, and such that $J_1$ contains $\theta_0$. For each $k$, by definition we have
\[
h_\a(\theta) = \mathrm{Re}\bigl(a_\a e^{2\pi i q_\a\theta}\bigr) \neq 0
\]
for all roots $\a$ and all $\theta\in J_k$. Moreover, the sign of this function is constant on $J_k$, hence determines a map
\[
\sigma_k:\Phi\to\{\pm1\}.
\]
We show that this map comes from a unique Weyl chamber. By construction, for any $\theta\in J_k$ we have
\[
r^{-q_\a} \mathrm{Re}\bigl(\langle\a,\widetilde f(re^{2\pi i\theta})\rangle\bigr) \to \mathrm{Re}\bigl(a_\a e^{2\pi i q_\a\theta}\bigr)
\]
as $r\to0$. Since $r^{-q_\a}$ is real and positive, the sign is not
affected by it. Hence
\[
\sigma_k(\a) = \sgn\!\left( \mathrm{Re}\bigl(\langle\a,\widetilde f(re^{2\pi i\theta})\rangle\bigr) \right)
\]
for sufficiently small $r>0$. Therefore $\sigma_k$ determines the Weyl
chamber containing the regular point
\[
x_{\theta,k,r}
:=
\mathrm{Re}\bigl(\widetilde f(re^{2\pi i\theta})\bigr).
\]
This gives the sequence of chambers $C_1,\ldots,C_r$. Choosing another lift of $f$ only changes the sequence by some $w\in W$ acting diagonally. In particular, the relative positions remain unchanged.

The second statement about the conjugacy classes containing the respective positive braids follows from the usual description of the braid group in terms of galleries, cf.~\cite{Deligne}.
\end{proof}
\begin{remark}
\begin{enumerate}
\item The dependence on $\theta_0$ is mild. It only fixes a starting sector, and any other choice is related to this choice by a cyclic permutation of $(C_1,\ldots,C_r)$.
\item When $f \in \frc^{\rs}(F) \cap \frc(\CC[t^{-1}])$ without constant term, the above construction is already contained in \cite[\S 4.3.2]{BBMY}. 

\end{enumerate}
\end{remark}
The following is now an immediate consequence of Proposition \ref{prop:root-valuations-determine-braid} and Corollary \ref{cor:conjinv}.
\begin{cor} \label{c:braidVar}
Let $f, f' \in \frc^{\rs}(F)$ with all root valuations positive (resp. negative). Assume that $f$ and $f'$ have the same root valuation datum. 
\begin{enumerate} 
\item The positive braid $\b_{f}^{+}$ is conjugate to $\b_{f'}^{+}$, and consequently $\chi(\Delta(\b_{f}^{+})) \cong \chi(\Delta(\b_{f'}^{+}))$ as well as $\chi(\nabla(\b_{f}^{+})) \cong \chi(\nabla(\b_{f'}^{+}))$. Therefore the (compactly supported) cohomology of $Z(\b_{f}^{+})$ is isomorphic to that of $Z(\b_{f'}^{+})$. The same is true for $X(\b_{f}^+,g)$ and $X(\b_{f'}^+,g)$, and the equivariant versions.
\item In the case $\b_f$ is pure, we have 
\[ \chi(\Delta(\b_{f}^{+})) \cong \chi(\Delta(\pi_{W/W_1}^{d_1} \pi_{W_1/W_2}^{d_2} \dots \pi_{W_{r-2}/W_{r-1}}^{d_{r-1}}\pi_{W_{r-1}}^{d_{r}})), \]
for $d_j$ the root valuations of $f$ and associated filtration of $W$, and similarly for $\nabla$. 
\end{enumerate}
\end{cor}

\appendix
\section{Local triviality near compact subsets} \label{app:local-triviality}
In this appendix we prove a local triviality result for submersions near compact subsets of a fiber. We expect this result to be standard in differential topology, but were unable to find a reference in the precise form needed here, so we include a proof.

\begin{lemma}\label{lem:loctrivcompact}
Let $f : M \to B$ be a smooth submersion of connected smooth real manifolds, $b\in B$, and $K_b \subset f^{-1}(b) = M_{b}$ be a compact subset. Then there are neighborhoods $U_b \subset M_{b}$ of $K_b$, $V \subset B$ of $b$ and $U\subset M$ of $K_b$ together with a diffeomorphism 
\[ U_b \times V \xrightarrow{\sim} U\]
over $V$.
\end{lemma}
\begin{proof}
Since $f$ is a submersion, the differential $df : TM \to TB$ is surjective. Let $TM_{v} = \ker(df)$ be the vertical subbundle, and choose a splitting $TM = TM_{v} \oplus H$ (for example by choosing some Riemannian metric on $M$). Then $H$ is a horizontal distribution, allowing us to lift paths locally. We will use the distribution to construct a ``radial'' neighborhood of $K_b$ restricted to which $f$ trivializes.

Let $V \subset B$ be an open neighborhood of $b$ with local coordinates $(t_1,...,t_n)$ such that $b$ is given by $t_1=...=t_n=0$. Then, for sufficiently small $t=(t_1,...,t_n) \in V$ the idea is to lift the path 
\[ [0,1] \to V, s \mapsto (st_1,...,st_n). \]

The horizontal distribution allows us to lift the coordinate vector fields $\frac{\partial}{\partial t_i}$ uniquely to horizontal vector fields $\xi_i$ on $W := f^{-1}(V)$. For the parameter $t=(t_1,...,t_n) \in V$ consider the vector field 
\[ \xi_t := \sum_{i=1}^{n} t_i \xi_i.\]
We may think of this as a vector field $\widehat{\xi}$ on $W \times V$ by putting $\widehat{\xi}(x,t) = (\xi_t(x), 0)$. Consider its maximal flow
\[ \widehat{\phi} : \widehat{D} \to W \times V \] 
where $\widehat{D} \subset \RR \times W \times V$ is the open flow domain. In this situation
\[ \widehat{\phi}(s,x,t) = (\phi_t(s,x),t) \]
where $\phi_t$ is the flow of $\xi_t$. In this way, we get a flow 
\[\phi = \pr_{W}\circ \widehat{\phi}: \widehat{D} \to W, \]
with parameter $t \in V$.

Because $K_b$ is compact, we may choose a relatively compact open neighborhood $K_b \subset U_b \subset M_b$. Let $C$ be its compact closure.

Note that $\xi_0 = 0$, so the flow for $t=0$ exists for all time. This means 
\[  [0,1] \times C \times \{0\} \subset [0,1] \times W \times \{0\} \subset \widehat{D} \]
is contained in the flow domain.

Then, for each point $d = (s,x,0) \in  [0,1] \times C \times \{0\}$, by openness of $\widehat{D}$, we find open neighborhoods $I_d \subset \RR$ of $s$, $U_d \subset W$ of $x$ and $V_d$ of $0$ such that
\[ I_d \times U_d \times V_d \subset \widehat{D}. \]
By compactness of $[0,1] \times C$, we can find finitely many points $d_1,\ldots,d_r$ such that the corresponding neighborhoods cover $[0,1] \times C \times \{0\}$. Let $V' = \bigcap V_{d_j}$, which is non-empty and open. Then 
\[[0,1] \times C \times V' \subset \widehat{D}.\]
In other words, for parameters $t \in V'$, the flow exists at all times $s \in [0,1]$. In particular, we can restrict to $U_b \subset C$ to define the map
\begin{align*} \Phi : U_b \times V' &\to W  \\
(x,t) &\mapsto \phi(1,x,t). 
\end{align*}
Then by construction of the flow, this map satisfies $(f \circ \Phi)(x,t) = \pr_{V'}(x,t) = t$. From here, one sees that $\Phi$ is injective. Indeed, if $\Phi(x,t)=\Phi(x,t')$, then applying $f$ we find that $t=t'$. Then we see that
\[\Phi(x,t) = \phi(1,x,t) = \phi(1,x',t), \]
so by injectivity of the flow map $x \mapsto \phi_t^1(x)$ we find that $x=x'$. After possibly shrinking $V'$ further, the map $\Phi$ is a diffeomorphism over $V'$ onto its image $U:=\Phi(U_b \times V')$. This $U$ is the desired radial neighborhood of $K_b$.
\end{proof}
By parallel transport for the map $U \to V$ we obtain the following statement.
\begin{cor}
In the situation of the above lemma, let $b, b'\in V$, and let $U_{b'}$ be the fiber of $U\to V$ over $b'$. Then any choice of trivialization $U_b \times V \cong U$ of $U\to V$ induces an isomorphism $U_b \cong U_{b'}$, hence an isomorphism
\[\pi_{1}(U_b) \cong \pi_{1}(U_{b'}).\]
Moreover, there is an embedding $U_{b'} \to M_{b'}$, so we obtain 
\[\pi_{1}(U_b) \to \pi_{1}(M_{b'}). \]
\end{cor}

The following corollary now follows easily. 

\begin{cor}\label{cor:centloop}
Let $f : M \to \RR$ be a smooth submersion, let $C$ be compact and suppose
\[ P : C \times [0,\d) \to M \]
is continuous and satisfies $f(P(c,t)) = t$. Write $p_t(c) := P(c,t), K_0 := p_0(C) \subset M_0$. Let $\Psi : U_0 \times V \cong U$ be any local trivialization over a neighborhood $V$ of $0$ with $K_0 \subset U_0$ and $\Psi(x,0) = x$ for $x\in U_0$. Then, after decreasing $\d$, the following hold:
\begin{enumerate}
\item We have $p_t(C) \subset U_t := M_t \cap U$ for every sufficiently small $t \ge 0$.
\item Let $p_{t}^{\Psi} : C \to U_t$ be the parallel transport of $p_0$ defined by $\Psi$. Then $p_{t}^{\Psi}$ and $p_{t}$ are homotopic via maps $C\to U_t$.
\item The homotopy class of $p_{t}^{\Psi} : C \to M_t$ is independent of the chosen local trivialization.
\end{enumerate}
Moreover, if 
\[ Q : C \times [0,\d) \to M \] is another continuous map satisfying $f(Q(c,t)) = t$, $Q(c,0) = p_0(c)$, then for $q_t(c) := Q(c,t)$ the maps $p_t$ and $q_t$ are homotopic after decreasing $\d$ if necessary. 
\end{cor}
\begin{proof}
We prove (1). Since $U$ is an open neighborhood of $K_0$, we have that $P^{-1}(U) \subset C \times [0,\d)$ is an open neighborhood of $C \times \{0\}$. By compactness of $C$ and the tube lemma we may assume
\[ P(C\times [0,\delta)) \subset U \]
at the expense of shrinking $\d$, proving (1).

We prove (2). Define the map $a := \pr_1 \circ \Psi^{-1} \circ P : C \times [0,\d) \to U_0$. Then $\Psi^{-1}(p_t(c)) = (a(c,t),t)$, and $a(c,0) = p_0(c)$. The parallel transport of $p_0$ defined by $\Psi$ is the map 
\[ p_t^{\Psi}(c) = \Psi(p_0(c),t). \]
We define the desired homotopy as follows
\[ H_t : C \times [0,1] \to U_t, \qquad H_t(c,s) := \Psi(a(c,st),t).\]
Since $\Psi$ is a map over $V$, the image of $H_t$ is contained in the fiber $U_t$. Moreover, we have
\[ H_t(c,0) = \Psi(a(c,0),t) = \Psi(p_0(c),t) = p_t^{\Psi}(c).\]
On the other hand, $H_t(c,1) = p_t(c)$, and $H_t$ is the desired homotopy. This proves (2).

We prove (3). Assume
\[ \Psi' : U_0' \times V' \cong U' \]
is another local trivialization near $K_0$. After decreasing $\d$ further, the above argument applies to $\Psi$ and $\Psi'$. Therefore, by (2), we have a chain of homotopies  $p_t^{\Psi} \simeq p_t \simeq  p_t^{\Psi'}$. This proves (3).

The final statement follows similarly: since $q_0 = p_0$, their parallel transports agree, and hence $p_t \simeq p_t^{\Psi} = q_t^{\Psi} \simeq q_t$.
\end{proof}

\end{document}